\documentclass{siamltex}%
\usepackage{amsfonts}
\usepackage{amsmath}
\usepackage{amssymb}
\usepackage{graphicx}
\usepackage{enumerate}
\usepackage{color}
\usepackage{hyperref}%
\newtheorem{example}[theorem]{Example}

\newtheorem{remark}[theorem]{Remark}

\begin{document}

\title{Controllability of periodic linear systems, the Poincar\'{e} sphere, and
quasi-affine systems}
\author{Fritz Colonius\\Institut f\"{u}r Mathematik, Universit\"{a}t Augsburg, Augsburg, Germany
\and Alexandre J. Santana and Juliana Setti\\Departamento de Matem\'{a}tica, Universidade Estadual de Maring\'{a}\\Maring\'{a}, Brazil}
\maketitle

\textbf{Abstract: }For periodic linear control systems with bounded control
range, an autonomized system is introduced by adding the phase to the state of
the system. Here a unique control set (i.e., a maximal set of approximate
controllability) with nonvoid interior exists. It is determined by the
spectral subspaces of the homogeneous part which is a periodic linear
differential equation. Using the Poincar\'{e} sphere one obtains a
compactification of the state space allowing us to describe the behavior
\textquotedblleft near infinity\textquotedblright\ of the original control
system. Furthermore, an application to quasi-affine systems yields a unique
control set with nonvoid interior.

\textbf{Key words. }periodic linear control system, quasi-affine control
system, control set, Floquet theory

\textbf{AMS subject classification.} 93B05, 34H05

\section{Introduction}

We study controllability properties for periodic linear control systems and
give an application to quasi-affine control systems. Periodic linear control
systems have the form
\begin{equation}
\dot{x}(t)=A(t)x(t)+B(t)u(t),\quad u(t)\in U, \label{periodic0}%
\end{equation}
where $A\in L^{\infty}(\mathbb{R},\mathbb{R}^{d\times d})$ and $B\in
L^{\infty}(\mathbb{R},\mathbb{R}^{d\times m})$ are $T$-periodic for some
$T>0$. We suppose that the controls $u=(u_{1},\ldots,u_{m})$ have values in a
bounded convex neighborhood $U$ of the origin in $\mathbb{R}^{m}$. The set of
admissible controls is
\[
\mathcal{U}=\{u\in L^{\infty}(\mathbb{R},\mathbb{R}^{m})\left\vert u(t)\in
U\text{ for almost all }t\right.  \}.
\]
We denote the solutions (in the Carath\'{e}odory sense) of (\ref{periodic0})
with initial condition $x(t_{0})=x_{0}$ by $\varphi(t;t_{0},x_{0}%
,u),\,t\in\mathbb{R}$. The homogeneous part of (\ref{periodic0}) is the
(uncontrolled) homogeneous periodic differential equation
\begin{equation}
\dot{x}(t)=A(t)x(t). \label{hom}%
\end{equation}
Nonautonomous control systems can be autonomized by including time in the
state of the system. This is useful, if recurrence properties can be
exploited; cf. Johnson and Nerurkar \cite{JoNer92}. In the $T$-periodic case,
it suffices to add the phases $\tau\in\lbrack0,T)$ to the states in
$\mathbb{R}^{d}$ (cf. Gayer \cite{Gayer05} for general periodic nonlinear
systems) and we follow this approach. We give a spectral characterization of
the reachable sets generalizing Sontag \cite[Corollary 3.6.7]{Son98} for
autonomous linear control systems; cf. also \cite[p. 139]{Son98} for some
historical remarks. The proof also uses arguments from Colonius, Cossich, and
Santana \cite[Theorem 15]{ColCS22} for autonomous discrete-time systems. This
yields a characterization of the unique control set (i.e., a maximal set of
approximate controllability) with nonvoid interior. The Poincar\'{e} sphere
from the global theory of nonlinear differential equations (introduced by
Poincar\'{e} \cite{Poin} for polynomial differential equations) provides a
compactification of the state space; cf. the monograph Perko \cite[Section
3.10]{Perko} and, e.g., Valls \cite{Valls22} for a recent contribution. This
leads us to a description\ of the behavior \textquotedblleft near
infinity\textquotedblright\ of the original control system. Since the induced
system on the Poincar\'{e} sphere is obtained by projection of a homogeneous
system it suffices to consider its restriction to the upper hemisphere.
Alternatively one might consider the induced system on projective space. In
\cite{ColSS22} we have used the latter approach for autonomous affine control
systems. We remark that Da Silva \cite{DaS16} has generalized \cite[Corollary
3.6.7]{Son98} in another direction, for linear control systems on solvable Lie
groups. General background on control of periodic linear systems is contained
in Bittanti and Colaneri \cite{Bitt}.

The present paper may also be considered as a contribution to a Floquet theory
of periodic control systems. They involve two $T$-periodic matrix functions
$A(\cdot)$ and $B(\cdot)$ and a periodic coordinate change can transform only
one of them to a constant matrix, hence periodic linear systems cannot be
conjugated to autonomous linear systems. But the formulation of Floquet theory
in the framework of linear skew product flows can be generalized (cf., e.g.,
Colonius and Kliemann \cite[Chapter 7]{ColK14}, and Kloeden and Rasmussen
\cite{KR11} for the general theory of skew product flows). The spectral
subspaces (the stable, center, and unstable subspaces) of (\ref{hom})
depending on the phase $\tau\in\lbrack0,T)$ characterize controllability properties.

In the\ last part of this paper we introduce quasi-affine control systems
which have the form%
\begin{equation}
\dot{x}(t)=A(v(t))x(t)+B(v(t))u(t), \label{qaffine1}%
\end{equation}
with $A(v):=A_{0}+\sum_{i=1}^{p}v_{i}A_{i}$ for $v\in V\subset\mathbb{R}^{p}$,
where $A_{0},A_{1},\ldots,A_{p}\in\mathbb{R}^{d\times d}$, and $B:V\rightarrow
\mathbb{R}^{d\times m}$ is continuous. The controls $(u,v)$ have values in a
compact convex neighborhood $U\times V\subset\mathbb{R}^{m}\times
\mathbb{R}^{p}$ of $(0,0)$, and the set of admissible controls is
\[
\mathcal{U}\times\mathcal{V}=\{(u,v)\in L^{\infty}(\mathbb{R},\mathbb{R}%
^{m})\times L^{\infty}(\mathbb{R},\mathbb{R}^{p})\left\vert u(t)\in U\text{
and }v(t)\in V\text{ for almost all }t\right.  \}.
\]
Quasi-affine systems look similar to linear control systems but the
coefficient matrices in front of $x$ and $u$ may depend on the additional
controls $v$. If a periodic $v\in\mathcal{V}$ is fixed, one obtains a periodic
linear control system with controls $u$. We use this relation to prove results
for control sets of quasi-affine systems. A special case are affine control
systems with separated additive and multiplicative control terms,%
\begin{equation}
\dot{x}(t)=A_{0}x(t)+\sum_{i=1}^{p}v_{i}(t)A_{i}x(t)+Bu(t). \label{affine}%
\end{equation}
Controllability properties of affine systems are a classical topic in control.
We only refer to the monographs Mohler \cite{Mohler}, Elliott \cite{Elliott},
and Jurdjevic \cite{Jurd97}. Our recent paper \cite{ColSS22} proves results on
control sets of general affine systems; cf. also \cite{ColRS} for control sets
about equilibria.

The contents of this paper are the following. After preliminaries in Section
\ref{Section2} on $T$-periodic linear control systems, Section \ref{Section3}
introduces the autonomized control system with state space $\mathbb{S}%
^{1}\times\mathbb{R}^{d}$, where the unit circle $\mathbb{S}^{1}$ is
parametrized by $\tau\in\lbrack0,T)$. In Section \ref{Section4}, Theorem
\ref{Theorem_sub} characterizes the reachable and controllable subsets using
the spectral subbundles of the periodic differential equation (\ref{hom}).
Theorem \ref{Theorem_cs1} shows that a unique control set $D^{a}%
\subset\mathbb{S}^{1}\times\mathbb{R}^{d}$ with nonvoid interior exists with
unbounded part given by the center subbundle. Section \ref{Section5} projects
the control system to the open upper hemisphere $\mathbb{S}^{d,+}$ of the
Poincar\'{e} sphere. Together with the equator $\mathbb{S}^{d,0}$ this
constitutes a compactification where the behavior \textquotedblleft near
infinity\textquotedblright\ is mapped onto the behavior near the equator. The
control set $D^{a}$ on $\mathbb{S}^{1}\times\mathbb{R}^{d}$ projects onto the
control set $D_{P}^{a}$ on $\mathbb{S}^{1}\times\mathbb{S}^{d,+}$ and the
intersection of $\overline{\mathrm{int}D_{P}^{a}}$ with $\mathbb{S}^{1}%
\times\mathbb{S}^{d,0}$ is determined by the image of the center subbundle of
(\ref{hom}). These results are also new for autonomous linear control systems.
Section \ref{Section6} presents some low dimensional examples, and, finally,
Section \ref{Section7} introduces quasi-affine systems. Theorem
\ref{Theorem7.2} characterizes their unique control set with nonvoid interior
using the control sets of the periodic linear control systems for fixed
periodic $v\in\mathcal{V}$.

\textbf{Notation:} For a matrix $A\in\mathbb{R}^{d\times d}$ the set of
eigenvalues is denoted by $\mathrm{spec}(A)$ and the real generalized
eigenspace for $\mu\in\mathrm{spec}(A)$ is $GE(A,\mu)$. The $d\times d$
identity matrix is $I_{d}$ and $\mathbb{N}=\{0,1,2,\ldots\}$. The interior of
a set $M$ in a metric space is $\mathrm{int}M$.

\section{Preliminaries\label{Section2}}

In this section we introduce some notation and discuss consequences of the
$T$-periodicity property.\ In particular, we recall a result on
controllability for periodic systems without control restrictions.

The principal fundamental solution $X(t,s)\in\mathbb{R}^{d\times d}%
,s,t\in\mathbb{R}$, of the homogeneous equation (\ref{hom}) is the matrix
solution of
\[
\frac{d}{dt}X(t,s)=A(t)X(t,s)\text{ with }X(s,s)=I_{d}.
\]
Here $X(t,r)X(r,s)=X(t,s),t,r,s\in\mathbb{R}$, and by $T$-periodicity
$X(t+kT,s+kT)=X(t,s)$ for all $k\in\mathbb{Z}$. The variation-of-parameters
formula for the solutions of (\ref{periodic0}) yields%
\begin{equation}
\varphi(t;t_{0},x,u)=X(t,t_{0})x+\int_{t_{0}}^{t}X(t,s)B(s)u(s)ds. \label{VdP}%
\end{equation}
Denote for $x\in\mathbb{R}^{d}$ the reachable set for $t\geq t_{0}$ and the
controllable set for $t\leq t_{0}$ of (\ref{periodic0}) by
\begin{equation}
\mathbf{R}_{t}(t_{0},x)=\{\varphi(t;t_{0},x,u)\left\vert u\in\mathcal{U}%
\right.  \},\mathbf{C}_{t}(t_{0},x)=\{y\in\mathbb{R}^{d}\left\vert \exists
u\in\mathcal{U}:\varphi(t_{0};t,y,u)=x\right.  \}, \label{C1}%
\end{equation}
resp., and let the reachable set and the controllable set be%
\[
\mathbf{R}(t_{0},x):=\bigcup\nolimits_{t\geq t_{0}}\mathbf{R}_{t}%
(t_{0},x)\text{ and }\mathbf{C}(t_{0},x):=\bigcup\nolimits_{t\leq t_{0}%
}\mathbf{C}_{t}(t_{0},x).
\]

\begin{lemma}
\label{Lemma1}Let $x\in\mathbb{R}^{d},$ $t\geq t_{0}$, and $k\in\mathbb{N}$.

(i) The reachable sets are convex and satisfy $\mathbf{R}_{t}(t_{0}%
,x)=\mathbf{R}_{t+kT}(t_{0}+kT,x)$.

(ii) The reachable sets $\mathbf{R}_{kT+t_{0}}(t_{0},0)$ are increasing with
$k\in\mathbb{N}$.
\end{lemma}

\begin{proof}
Convexity of $\mathbf{R}_{t}(t_{0},x)$ holds since the control range $U$ is
convex. The equality in (i) follows from
\begin{align*}
\varphi(t+kT;t_{0}+kT,x,u)  &  =X(t+kT,t_{0}+kT)x+\int_{t_{0}+kT}%
^{t+kT}X(t+kT,s)B(s)u(s)ds\\
&  =X(t,t_{0})x+\int_{t_{0}}^{t}X(t+kT,s+kT)B(s+kT)u(s+kT)ds\\
&  =X(t,t_{0})x+\int_{t_{0}}^{t}X(t,s)B(s)u(s+kT)ds\\
&  =\varphi(t;t_{0},x,u(\cdot+kT)),
\end{align*}
where $u(\cdot+kT)(s):=u(s+kT),s\in\mathbb{R}$.

For assertion (ii) let $k\geq\ell$ and consider $x\in\mathbf{R}_{\ell T+t_{0}%
}(t_{0},0)$ with%
\[
x=\varphi(\ell T+t_{0};t_{0},0,u)=\int_{t_{0}}^{\ell T+t_{0}}X(\ell
T+t_{0},s)B(s)u(s)ds.
\]
Define
\[
v(t)=\left\{
\begin{array}
[c]{lll}%
0 & \text{for} & t\in\lbrack t_{0},(k-\ell)T+t_{0})\\
u(t-(k-\ell)T) & \text{for} & t\in\lbrack(k-\ell)T+t_{0},kT+t_{0}]
\end{array}
\right.  .
\]
Then one obtains
\begin{align*}
&  \varphi(kT+t_{0};t_{0},0,v)=\int_{t_{0}}^{kT+t_{0}}X(kT+t_{0}%
,s)B(s)v(s)ds\\
&  =\int_{t_{0}}^{(k-\ell)T+t_{0}}X(kT+t_{0},s)B(s)v(s)ds+\int_{(k-\ell
)T+t_{0}}^{kT+t_{0}}X(kT+t_{0},s)B(s)v(s)ds\\
&  =0+\int_{t_{0}}^{\ell T+t_{0}}X(kT+t_{0},s+(k-\ell)T)B(s+(k-\ell
)T)v(s+(k-\ell)T)ds\\
&  =\int_{t_{0}}^{\ell T+t_{0}}X(\ell T+t_{0},s)B(s)u(s)ds=x,
\end{align*}
hence $\mathbf{R}_{\ell T+t_{0}}(t_{0},0)\subset\mathbf{R}_{kT+t_{0}}%
(t_{0},0)$.
\end{proof}

In order to clarify the relationship between the reachable and the
controllable sets of the considered nonautonomous control systems it is
convenient to introduce the following time reversed systems (cf. Sontag
\cite[Definition 2.6.7 and Lemma 2.6.8]{Son98}). The reversal of
(\ref{periodic0}) at $\mu\in\mathbb{R}$ is
\begin{equation}
\dot{y}(t)=-A(\mu-t)y(t)-B(\mu-t)u(t),\quad u\in\mathcal{U}, \label{reversed}%
\end{equation}
with trajectories denoted by $\varphi_{\mu}^{-}(t;t_{0},x_{0},u),\,t\in
\mathbb{R}$.

\begin{lemma}
\label{Lemma2}For $t_{1}<t_{0}$ the controllable set $\mathbf{C}_{t_{1}}%
(t_{0},x)$ of (\ref{periodic0}) coincides with the reachable set
$\mathbf{R}_{t_{0}}^{-}(t_{1},x)$ of the time-reversed system (\ref{reversed})
at $\mu=t_{0}+t_{1}$.
\end{lemma}

\begin{proof}
Let $y\in\mathbf{C}_{t_{1}}(t_{0},x)$ with $\varphi(t_{0};t_{1},y,u)=x$.
Define the control $u^{-}(t):=u(t_{0}+t_{1}-t),t\in\mathbb{R}$. The function
$y(t):=\varphi(t_{0}+t_{1}-t;t_{1},y,u),t\in\lbrack t_{1},t_{0}]$, satisfies
the differential equation%
\[
\dot{y}(t)=-A(t_{0}+t_{1}-t)y(t)-B(t_{0}+t_{1}-t)u(t_{0}+t_{1}-t),
\]
and $y(t_{0})=y$ and $y(t_{1})=x$. Thus $y(t)=\varphi_{t_{0}+t_{1}}%
^{-}(t;t_{1},x,u^{-}),t\in\lbrack t_{1},t_{0}]$, and the assertion follows.
\end{proof}

Since $\mathbf{C}_{-kT+t_{0}}(t_{0},x)=\mathbf{R}_{t_{0}}^{-}(-kT+t_{0}%
,x)=\mathbf{R}_{kT+t_{0}}^{-}(t_{0},x)$ Lemma \ref{Lemma1} implies that also
the controllable sets are convex and for $k\geq\ell$ in $\mathbb{N}$ the
inclusion $\mathbf{C}_{-\ell T+t_{0}}(t_{0},0)\subset\mathbf{C}_{-kT+t_{0}%
}(t_{0},0)$ holds.

\begin{proposition}
\label{Proposition9_JDDE}For $\tau\in\left[  0,T\right]  $ consider
$x\in\mathbf{R}_{kT+\tau}(\tau,0)$ and $y\in\mathbf{R}_{\ell T+\tau}(\tau,0)$
where $k,\ell\in\mathbb{N}$. Then it follows that%
\begin{equation}
x+X(kT+\tau,\tau)y=x+X(T+\tau,\tau)^{k}y\in\mathbf{R}_{(k+\ell)T+\tau}%
(\tau,0). \label{2.6}%
\end{equation}

\end{proposition}

\begin{proof}
There are $u,v\in\mathcal{U}$ with
\begin{align*}
x  &  =\varphi(kT+\tau;\tau,0,u)=\int_{\tau}^{kT+\tau}X(kT+\tau
,s)B(s)u(s)ds,\\
y  &  =\varphi(\ell T+\tau;\tau,0,v)=\int_{\tau}^{\ell T+\tau}X(\ell
T+\tau,s)B(s)v(s)ds.
\end{align*}
Define%
\[
w(t)=\left\{
\begin{array}
[c]{lll}%
v(t) & \text{for} & t\in\lbrack\tau,\ell T+\tau]\\
u(t-\ell T) & \text{for} & t\in(\ell T+\tau,(k+\ell)T+\tau]
\end{array}
\right.  .
\]
Then one computes%
\begin{align*}
&  \varphi((k+\ell)T+\tau;\tau,0,w)=\int_{\tau}^{(k+\ell)T+\tau}%
X((k+\ell)T+\tau,s)B(s)w(s)ds\\
&  =\int_{\tau}^{\ell T+\tau}X((k+\ell)T+\tau,s)B(s)v(s)ds\\
&  \qquad+\int_{\ell T+\tau}^{(k+\ell)T+\tau}X((k+\ell)T+\tau,s)B(s)u(s-\ell
T)ds\\
&  =\int_{\tau}^{\ell T+\tau}X((k+\ell)T+\tau,\ell T+\tau)X(\ell
T+\tau,s)B(s)v(s)ds\\
&  \qquad+\int_{\tau}^{kT+\tau}X((k+\ell)T+\tau,\ell T+s)B(\ell T+s)u(s)ds\\
&  =X(kT+\tau,\tau)\int_{\tau}^{\ell T+\tau}X(\ell T+\tau,s)B(s)v(s)ds+\int
_{\tau}^{kT+\tau}X(kT+\tau,s)B(s)u(s)ds\\
&  =X(kT+\tau,\tau)y+x.
\end{align*}
Thus (\ref{2.6}) holds.
\end{proof}

Controllability criteria for periodic linear systems without control
constraints are well known. The following theorem is due to Brunovsky
\cite{Brun}, slightly reformulated.

\begin{theorem}
\label{Theorem_Brunovsky}For the periodic linear system in (\ref{periodic0})
without control restrictions, the following properties are equivalent:

(i) For any two points $x_{1},x_{2}\in\mathbb{R}^{d}$ and any $t_{0}%
\in\mathbb{R}$ there are $t_{1}>t_{0}$ and $u\in L^{\infty}([t_{0}%
,t_{1}],\mathbb{R}^{m})$ such that $\varphi(t_{1};t_{0},x_{1},u)=x_{2}$.

(ii) For any two points $x_{1},x_{2}\in\mathbb{R}^{d}$ and $\tau\in\left[
0,T\right]  $ there is $u\in L^{\infty}([\tau,dT+\tau],\mathbb{R}^{m})$ such
that $\varphi(dT+\tau;\tau,x_{1},u)=x_{2}$.

(iii) The rows of the matrix function $X(t,0)^{-1}B(t),t\in\lbrack0,dT]$, are
linearly independent.

If any of the equivalent conditions above is satisfied, the system in
(\ref{periodic0}) without control restrictions is called controllable.
\end{theorem}

\begin{proof}
Brunovsky \cite[Proposition 3]{Brun} shows that conditions (i) and (iii) are
equivalent. By\ $T$-periodicity, condition (ii) implies (i). Conversely, cf.
the proof of \cite[Proposition 3]{Brun}, condition (i) implies (ii) for
$\tau=0$. If (i) holds for the system with matrix functions $A(t)$ and
$B(t),t\in\mathbb{R}$, it also holds for the system with $A(\tau+t)$ and
$B(\tau+t),t\in\mathbb{R}$, with $\tau\in\left[  0,T\right]  $. Hence
condition (ii) follows for all $\tau\in\left[  0,T\right]  $.
\end{proof}

\begin{remark}
Condition (iii) above generalizes the Kalman-criterion for controllability of
autonomous systems. It is equivalent to assignability of the spectrum by
$T$-periodic state feedbacks, \cite[Theorem on p. 302]{Brun}. As shown in
Bittanti, Guarbadassi, Mafezzoni, and Silverman \cite{Bit78} and Bittanti,
Colaneri, and Guarbadassi \cite{Bit84} a criterion generalizing the
Hautus-Popov spectral characterization for controllability is only equivalent
to null-controllability.
\end{remark}

Theorem \ref{Theorem_Brunovsky} implies the following first result on
controllability properties of the system with control restrictions.

\begin{proposition}
\label{proposition_R}Assume that the periodic linear system in
(\ref{periodic0}) without control restrictions is controllable, and let
$\tau\in\left[  0,T\right]  $. Then for system (\ref{periodic0}) with controls
$u\in\mathcal{U}$ the reachable set $\mathbf{R}_{dT+\tau}(\tau,0)$ and the
controllable set $\mathbf{C}_{-dT+\tau}(\tau,0)$ of (\ref{periodic0}) are
convex and contain an $\varepsilon$-ball $\mathbf{B}(0;\varepsilon)$ with
$\varepsilon>0$ around $0\in\mathbb{R}^{d}$. The sets%
\[
\mathbf{R}_{\mathbb{N}T+\tau}(\tau,0):=\bigcup_{k\in\mathbb{N}}\mathbf{R}%
_{kT+\tau}(\tau,0)\text{ and }\mathbf{C}_{-\mathbb{N}T+\tau}(\tau
,0):=\bigcup_{k\in\mathbb{N}}\mathbf{C}_{-kT+\tau}(\tau,0),\tau\in
\lbrack0,T],
\]
are convex and open. Furthermore also the sets%
\[
\mathbf{R}_{\mathbb{N}T+\tau}(0,0),\mathrm{int}\mathbf{R}_{\mathbb{N}T+\tau
}(0,0),\mathbf{C}_{-\mathbb{N}T+\tau}(0,0),\text{ and }\mathrm{int}%
\mathbf{C}_{-\mathbb{N}T+\tau}(0,0)
\]
are convex, and%
\[
\mathbf{R}_{\mathbb{N}T+\tau}(\tau,0)\subset\mathrm{int}\mathbf{R}%
_{\mathbb{N}T+\tau}(0,0)\text{ and }\mathbf{C}_{-\mathbb{N}T+\tau}%
(\tau,0)\subset\mathrm{int}\mathbf{C}_{-\mathbb{N}T+\tau}(0,0).
\]

\end{proposition}

\begin{proof}
Convexity of $\mathbf{R}_{\mathbb{N}T+\tau}(\tau,0)$ follows from Lemma
\ref{Lemma1}. Fix a basis $y_{1},\ldots,y_{d}$ of $\mathbb{R}^{d}$. By Theorem
\ref{Theorem_Brunovsky} for every $\tau\in\left[  0,T\right]  $ there are
$u_{i}^{\tau}\in L^{\infty}([0,(d+1)T],\mathbb{R}^{m})$ with%
\[
y_{i}=\varphi(dT+\tau;\tau,0,u_{i}^{\tau})=\int_{\tau}^{dT+\tau}%
X(dT+\tau,s)B(s)u_{i}^{\tau}(s)ds\text{ for }i=1,\ldots,d,
\]
and $\varphi(dT+\tau;\tau,0,u)$ depends continuously on $(\tau,u)\in\left[
0,T\right]  \times L^{\infty}([0,(d+1)T],\mathbb{R}^{m})$. Let $\varepsilon
_{0}>0$ be small enough such that $z_{1},\ldots,z_{d}$ form a basis of
$\mathbb{R}^{d}$ for any $z_{i}\in\mathbf{B}(y_{i};\varepsilon_{0})$. By
continuity, there is for every $\tau_{0}\in\left[  0,T\right]  $ a $\delta
_{0}>0$ such that $\varphi(dT+\tau;\tau,0,u_{i}^{\tau_{0}})\in\mathbf{B}%
(y_{i};\varepsilon_{0})$ for $\left\vert \tau-\tau_{0}\right\vert <\delta_{0}%
$. Now compactness of $\left[  0,T\right]  $ shows that there are finitely
many $\tau_{j}\in\left[  0,T\right]  $ such that $\varphi(dT+\tau;\tau
,0,u_{i}^{\tau_{j}}),i=1,\ldots,d$, form a basis of $\mathbb{R}^{d}$. By
linearity, there is $\alpha>0$ such that also $\varphi(dT+\tau;\tau,0,\alpha
u_{i}^{\tau_{j}})$ form a basis of $\mathbb{R}^{d}$ and $\alpha u_{i}%
^{\tau_{j}}\in\mathcal{U}$ for all $i,j$. This shows that there is a ball
$\mathbf{B}(0;\varepsilon)$ contained in $\mathbf{R}_{dT+\tau}(\tau
,0)\subset\mathbf{R}_{\mathbb{N}T+\tau}(\tau,0)$ for all $\tau\in\left[
0,T\right]  $.

Let $x=\varphi(kT+\tau;\tau,0,u)\in\mathbf{R}_{\mathbb{N}T+\tau}(\tau,0)$ for
some $k\in\mathbb{N}$ and $u\in\mathcal{U}$. The set $X(kT+\tau,\tau
)\mathbf{B}(0;\varepsilon)$ is open and Proposition \ref{Proposition9_JDDE}
implies that for each $y\in\mathbf{B}(0;\varepsilon)$%
\[
x+X(kT+\tau,\tau)y\in\mathbf{R}_{(k+d)T+\tau}(\tau,0)\subset\mathbf{R}%
_{\mathbb{N}T+\tau}\mathbf{(}\tau,0).
\]
This shows that $\mathbf{R}_{\mathbb{N}T+\tau}\mathbf{(}\tau,0)$ is open. The
control%
\[
v(t):=\left\{
\begin{array}
[c]{lll}%
0 & \text{for} & t\in\lbrack0,\tau)\\
u(t) & \text{for} & t\in\lbrack\tau,kT+\tau]
\end{array}
\right.
\]
yields $x=\varphi(kT+\tau;\tau,0,u)=\varphi(kT+\tau;0,0,v)\in\mathbf{R}%
_{\mathbb{N}T+\tau}(0,0)$, hence the inclusion $\mathbf{R}_{\mathbb{N}T+\tau
}(\tau,0)\subset\mathrm{int}\mathbf{R}_{\mathbb{N}T+\tau}(0,0)$ holds. For
convexity of $\mathbf{R}_{\mathbb{N}T+\tau}(0,0)$ let for $i=1,2$%
\[
y_{i}=\varphi(\tau;0,x_{i},u_{i})\in\mathbf{R}_{\mathbb{N}T+\tau}(0,0)\text{
with }x_{i}\in\mathbf{R}_{\mathbb{N}T}(0,0),u_{i}\in\mathcal{U}.
\]
Then linearity implies for $\alpha\in\lbrack0,1]$ that%
\begin{align*}
\alpha y_{1}+(1-\alpha)y_{2}  &  =\alpha\varphi(\tau;0,x_{1},u_{1}%
)+(1-\alpha)\varphi(\tau;0,x_{2},u_{2})\\
&  =\varphi(\tau;0,\alpha x_{1}+(1-\alpha)x_{2},\alpha u_{1}+(1-\alpha
)u_{2})\in\mathbf{R}_{\mathbb{N}T+\tau}(0,0).
\end{align*}
Since $\mathbf{R}_{\mathbb{N}T+\tau}(0,0)$ is convex also $\mathrm{int}%
\mathbf{R}_{\mathbb{N}T+\tau}(0,0)$ is convex; cf. Dunford and Schwartz
\cite[Theorem V.2.1]{DS}.

The assertions for the controllable sets follow by time reversal from Lemma
\ref{Lemma2} and Lemma \ref{Lemma1}(i).
\end{proof}

\section{The autonomized system\label{Section3}}

First some results from Floquet theory are recalled (cf. Chicone
\cite{Chic99}, Teschl \cite{Tes}, and Colonius and Kliemann \cite[Section
7.2]{ColK14}). Then we introduce the autonomized system.

Consider the unit circle $\mathbb{S}^{1}$ parametrized by $\tau\in\lbrack0,T)$
and define the shift%
\[
\theta:\mathbb{R}\times\mathbb{S}^{1}\rightarrow\mathbb{S}^{1},\,\theta
(t;\tau)=t+\tau\text{ }\operatorname{mod}T\text{ for }t\in\mathbb{R}%
,\tau\mathbb{\in S}^{1}.
\]
Here $\tau+t\operatorname{mod}T$ denotes the unique element $\tau
+t-kT\in\lbrack0,T)$ for some $k\in\mathbb{Z}$. Let $\psi(t;\tau_{0},x_{0})$
be the solution of (\ref{hom}) with initial condition $x(\tau_{0})=x_{0}$ and
define%
\begin{equation}
\Psi=(\theta,\psi):\mathbb{R}\times\mathbb{S}^{1}\times\mathbb{R}%
^{d}\rightarrow\mathbb{S}^{1}\times\mathbb{R}^{d},\,\Psi(t;\tau_{0}%
,x_{0})=(\theta(t;\tau_{0}),\psi(t;\tau_{0},x_{0})). \label{aug_hom}%
\end{equation}
Then $\Psi$ is a continuous dynamical system, a linear skew product flow, on
$\mathbb{S}^{1}\times\mathbb{R}^{d}$.

The Floquet multipliers of equation (\ref{hom}) are the eigenvalues $\mu$ of
\[
X(T+\tau,\tau)=X(T+\tau,T)X(T,0)X(0,\tau)=X(\tau,0)X(T,0)X(\tau,0)^{-1}%
,~\tau\in\lbrack0,T).
\]
The Floquet exponents are $\lambda_{j}:=\frac{1}{T}\log\left\vert
\mu\right\vert $ (the Floquet exponents as defined here are the real parts of
the Floquet exponents defined in \cite{Chic99} and \cite{Tes}). Note that
$\lambda_{j}<0$ if and only if $\left\vert \mu\right\vert <1$. The following
result is \cite[Theorem 7.2.9]{ColK14}.

\begin{theorem}
\label{Theorem7.1.7}Let $\Psi=(\theta,\psi):\mathbb{R}\times\mathbb{S}%
^{1}\times\mathbb{R}^{d}\longrightarrow\mathbb{S}^{1}\times\mathbb{R}^{d}$ be
the linear skew product flow associated with the $T$-periodic linear
differential equation (\ref{hom}). For each $\tau\in\mathbb{S}^{1}$ there
exists a decomposition%
\[
\mathbb{R}^{d}=L(\lambda_{1},\tau)\oplus\cdots\oplus L(\lambda_{\ell},\tau)
\]
into linear subspaces $L(\lambda_{j},\tau)$, called the Floquet (or Lyapunov)
spaces, with the following properties:

(i) The Floquet spaces have dimension $d_{j}:=\dim L(\lambda_{j},\tau)$
independent of $\tau\in\mathbb{S}^{1}$.

(ii) They are invariant under multiplication by the principal fundamental
matrix in the following sense:
\[
X(t+\tau,\tau)L(\lambda_{j},\tau)=L(\lambda_{j},\theta(t;\tau))\text{ for all
}t\in\mathbb{R}\text{ and }\tau\in\mathbb{S}^{1}.
\]

(iii) For every $\tau\in\mathbb{S}^{1}$ the Floquet (or Lyapunov) exponents
satisfy%
\[
\lambda(x,\tau):=\lim_{t\rightarrow\pm\infty}\frac{1}{t}\log\Vert\psi
(t;\tau,x)\Vert=\lambda_{j}\text{ if and only if }0\not =x\in L(\lambda
_{j},\tau).
\]

\end{theorem}

The Floquet space $L(\lambda_{j},\tau)$ is the direct sum of the real
generalized eigenspaces for all Floquet multipliers $\mu$ with $\frac{1}%
{T}\log\left\vert \mu\right\vert =\lambda_{j}$,%
\[
L(\lambda_{j},\tau)=\bigoplus\nolimits_{\mu}GE(X(T+\tau,\tau),\mu).
\]
Define for $\tau\in\lbrack0,T)$ the stable, the center, and the unstable
subspaces, resp., by%
\[
E_{\tau}^{-}=\bigoplus\nolimits_{\lambda_{j}<0}L(\lambda_{j},\tau),\,E_{\tau
}^{0}:=L(0,\tau),\text{ and }E_{\tau}^{+}:=\bigoplus\nolimits_{\lambda_{j}%
>0}L(\lambda_{j},\tau).
\]
Then $E_{\tau}^{\pm}=X(\tau,0)E_{0}^{\pm}$ and $E_{\tau}^{0}=X(\tau
,0)E_{0}^{0}$, and $\mathbb{S}^{1}\times\mathbb{R}^{d}$ splits into the
Whitney sum $\mathcal{E}^{-}\oplus\mathcal{E}^{0}\oplus\mathcal{E}^{+}$ of the
stable, the center, and the unstable subbundles%
\begin{equation}
\mathcal{E}^{\pm}=\left\{  (\tau,x)\in\mathbb{S}^{1}\times\mathbb{R}%
^{d}\left\vert \,x\in E_{\tau}^{\pm}\right.  \right\}  ,\text{ }%
\mathcal{E}^{0}=\left\{  (\tau,x)\in\mathbb{S}^{1}\times\mathbb{R}%
^{d}\left\vert \,x\in E_{\tau}^{0}\right.  \right\}  , \label{bundle1}%
\end{equation}
resp. We also introduce analogous subbundles for the center-stable subspaces
and the center-unstable subspaces given by%
\[
E_{\tau}^{-,0}:=\bigoplus\nolimits_{\lambda_{j}\leq0}L(\lambda_{j},\tau)\text{
and }E_{\tau}^{+,0}=\bigoplus\nolimits_{\lambda_{j}\geq0}L(\lambda_{j}%
,\tau),\,\tau\in\lbrack0,T)\text{, resp.}%
\]
Similarly as for periodic differential equations, it is convenient for linear
periodic control systems of the form (\ref{periodic0}) to extend the state
space by adding the phase $\tau\in\lbrack0,T)$ to the state in order to get an
autonomous system. We obtain the following autonomized control system on
$\mathbb{S}^{1}\times\mathbb{R}^{d}$,%
\begin{equation}
\dot{\tau}(t)=1\operatorname{mod}T,\quad\dot{x}(t)=A(\tau(t))x(t)+B(\tau
(t))u(t),\qquad u\in\mathcal{U}, \label{aug1}%
\end{equation}
with solutions%
\[
\varphi^{a}(t;(\tau_{0},x_{0}),u)=\left(  \tau_{0}+t\operatorname{mod}%
T,\varphi(\tau_{0}+t;\tau_{0},x_{0},u)\right)  ,\ t\in\mathbb{R}.
\]
Observe that (\ref{aug1}) is not a linear control system.

\begin{remark}
If the matrix functions $A(\cdot)$ and $B(\cdot)$ are merely measurable, the
general existence theory of ordinary differential equations does not apply to
equation (\ref{aug1}). Nevertheless, the solutions are well defined.
\end{remark}

Denote the reachable and controllable sets for $t\geq0$ of (\ref{aug1}) by%
\begin{align*}
\mathbf{R}_{t}^{a}(\tau,x)  &  =\left\{  \varphi^{a}(t;(\tau,x),u)\left\vert
u\in\mathcal{U}\right.  \right\}  ,\\
\mathbf{C}_{t}^{a}(\tau,x)  &  =\{(\sigma,y)\in\mathbb{S}^{1}\times
\mathbb{R}^{d}\left\vert \exists u\in\mathcal{U}:\varphi^{a}(t;(\sigma
,y),u)=(\tau,x)\right.  \},\\
\mathbf{R}^{a}(\tau,x)  &  =\bigcup\nolimits_{t\geq0}\mathbf{R}_{t}^{a}%
(\tau,x)\text{ and }\mathbf{C}^{a}(\tau,x)=\bigcup\nolimits_{t\geq0}%
\mathbf{C}_{t}^{a}(\tau,x),
\end{align*}
resp. The time reversed autonomous system is%
\begin{equation}
\dot{\tau}(t)=-1\operatorname{mod}T,\quad\dot{y}(t)=-A(\tau(t))y(t)-B(\tau
(t))u(t),\qquad u\in\mathcal{U}. \label{aug2}%
\end{equation}
The reachable sets $\mathbf{R}_{t}^{a,-}(\tau,x)$ of the time-reversed
autonomized system (\ref{aug2}) coincide with the controllable sets
$\mathbf{C}_{t}^{a}(\tau,x)$ of system (\ref{aug1}). Note the following
relation to the reachable and controllable sets defined in (\ref{C1}) for the
periodic system (\ref{periodic0}).

\begin{lemma}
\label{Lemma3.3}For $(\tau,x)\in\mathbb{S}^{1}\times\mathbb{R}^{d}$ and
$t\geq0$ the following assertions hold:%
\begin{align*}
\mathbf{R}_{t}^{a}(\tau,x)  &  =\left\{  (\tau+t\operatorname{mod}%
T,y)\left\vert y\in\mathbf{R}_{\tau+t}(\tau,x)\right.  \right\}  ,\\
\mathbf{C}_{t}^{a}(\tau,x)  &  =\{(\sigma,y)\in\mathbb{S}^{1}\times
\mathbb{R}^{d}\left\vert \sigma+t=\tau\operatorname{mod}T\text{ and }%
y\in\mathbf{C}_{\sigma}(\sigma+t,x)\right.  \},\\
\mathbf{R}^{a}(0,0)  &  =\{(\tau,x)\in\mathbb{S}^{1}\times\mathbb{R}%
^{d}\left\vert x\in\mathbf{R}_{\mathbb{NT+\tau}}(0,0)\right.  \},\\
\mathbf{C}^{a}(0,0)  &  =\{(\tau,x)\in\mathbb{S}^{1}\times\mathbb{R}%
^{d}\left\vert x\in\mathbf{C}_{\mathbb{\tau}}(kT,0),k\geq1\right.  \}.
\end{align*}

\end{lemma}

\begin{proof}
By definition one has $(\sigma,y)=\varphi^{a}(t;(\tau,x),u)\in\mathbf{R}%
_{t}^{a}(\tau,x)$ if and only if $\sigma=\tau+t\operatorname{mod}T$ and
$\varphi(\tau+t;\tau,x,u)=y$. This shows that $y\in\mathbf{R}_{\tau+t}%
(\tau,x)$ and the first assertion follows. By definition $(\sigma
,y)\in\mathbf{C}_{t}^{a}(\tau,x)$ means that $\sigma+t=\tau\operatorname{mod}%
T$ and $\varphi(\sigma+t;\sigma,y,u)=x$ for some $u\in\mathcal{U}$, hence
$y\in\mathbf{C}_{\sigma}(\sigma+t,x)$. Furthermore, one finds%
\begin{align*}
\mathbf{R}^{a}(0,0)  &  =\bigcup\nolimits_{t\geq0}\left\{
(t\operatorname{mod}T,x)\left\vert x\in\mathbf{R}_{t}(0,0)\right.  \right\}
=\{(\tau,x)\in\mathbb{S}^{1}\times\mathbb{R}^{d}\left\vert x\in\mathbf{R}%
_{\mathbb{N}T+\tau}(0,0)\right.  \},\\
\mathbf{C}^{a}(0,0)  &  =\bigcup\nolimits_{t\geq0}\{(\tau,x)\in\mathbb{S}%
^{1}\times\mathbb{R}^{d}\left\vert \tau+t=0\operatorname{mod}T\text{ and }%
x\in\mathbf{C}_{\tau}(\tau+t,0)\right.  \}.
\end{align*}

\end{proof}

In particular, Lemma \ref{Lemma3.3} shows that $\mathbf{R}^{a}(0,0)=\mathbb{S}%
^{1}\times X$ for a subset $X\subset\mathbb{R}^{d}$ if and only if
$X=\mathbf{R}_{\mathbb{NT+\tau}}(0,0)$ for all $\tau\in\mathbb{S}^{1}$. The
next lemma provides additional information about the reachable and
controllable sets for $x=0$ of the autonomized system.

\begin{lemma}
\label{Lemma3.4}If the periodic linear system in (\ref{periodic0}) without
control restrictions is controllable, then%
\[
\mathbb{S}^{1}\times\{0\}\subset\mathrm{int}\mathbf{R}^{a}(\tau,0)\cap
\mathrm{int}\mathbf{C}^{a}(\tau,0),~\tau\in\lbrack0,T).
\]

\end{lemma}

\begin{proof}
For $\tau\in\mathbb{S}^{1}$ Lemma \ref{Lemma3.3} implies%
\begin{align*}
\mathbf{R}^{a}(\tau,0)  &  =\bigcup_{t\geq0}\mathbf{R}_{t}^{a}(\tau
,0)=\bigcup_{t\geq0}\left\{  (\tau+t\operatorname{mod}T,x)\left\vert
x\in\mathbf{R}_{\tau+t}(\tau,0)\right.  \right\} \\
&  =\{(\sigma,x)\in\mathbb{S}^{1}\times\mathbb{R}^{d}\left\vert \sigma
\in\lbrack\tau,T),x\in\mathbf{R}_{\mathbb{N}T+\sigma}(\tau,0)\right.  \}\\
&  \qquad\cup\{(\sigma,x)\in\mathbb{S}^{1}\times\mathbb{R}^{d}\left\vert
\sigma\in\lbrack0,\tau),x\in\mathbf{R}_{kT+\sigma}(\tau,0),k\geq1\right.
\}\text{.}%
\end{align*}
By Proposition \ref{proposition_R} the set $\mathbf{R}_{dT+\tau}(\tau,0)$
contains an $\varepsilon$-ball $\mathbf{B}(0;\varepsilon)$ around $0$ for some
$\varepsilon>0$. For $y\in\mathbf{B}(0;\varepsilon)$ there is some
$u\in\mathcal{U}$ with $y=\varphi(dT+\tau;\tau,0,u)$. Define for $\sigma
\in\lbrack0,T]$%
\[
v(t):=\left\{
\begin{array}
[c]{lll}%
u(t) & \text{for} & t\in\lbrack\tau,dT+\tau)\\
0 & \text{for} & t\in\lbrack dT+\tau,(d+1)T+\sigma]
\end{array}
\right.  .
\]
With the invertible matrices $Y(\sigma):=X((d+1)T+\sigma,dT+\tau)$ it follows
that%
\[
\varphi((d+1)T+\sigma;\tau,0,v)=X((d+1)T+\sigma,dT+\tau)y=Y(\sigma)y,
\]
showing that $Y(\sigma)\mathbf{B}(0;\varepsilon)\subset\mathbf{R}%
_{(d+1)T+\sigma}(\tau,0)$. The matrices $Y(\sigma)$ and hence also their
singular values $0<\delta_{1}(\sigma)\leq\cdots\leq\delta_{d}(\sigma)$ depend
continuously on $\sigma\in\lbrack0,T]$ (cf. Sontag \cite[Corollary
A.4.4]{Son98}). In particular, the minimal singular values $\delta_{1}%
(\sigma)$ are bounded away from $0$, since $[0,T]$ is compact. Now recall the
geometric interpretation of the singular values of a matrix $A$ (cf. e.g.
Arnold \cite[p. 118]{Arno98}): $\delta_{i}$ is the length of the $i$-th
principal axis of the ellipsoid $A(\mathbb{S}^{d-1})$ obtained as the image of
the unit sphere $\mathbb{S}^{d-1}$ under the linear mapping $A$. It follows
that there is a ball $\mathbf{B}(0;\varepsilon_{0})$ with $\varepsilon_{0}>0$
contained in every set $Y(\sigma)\mathbf{B}(0;\varepsilon),\sigma\in
\lbrack0,T]$. This proves that $\mathbf{B}(0;\varepsilon_{0})$ is contained in
every set $\bigcup\nolimits_{k\geq1}\mathbf{R}_{kT+\sigma}(\tau,0)$ and it
follows that $\mathbb{S}^{1}\times\{0\}\subset\mathrm{int}\mathbf{R}^{a}%
(\tau,0)$.

The assertion for $\mathbf{C}^{a}(\tau,0)$ follows by time reversal.
\end{proof}

\section{Spectral characterization of reachable and controllable
sets\label{Section4}}

In this section we characterize the reachable and the controllable sets of the
autonomized system (\ref{aug1}) by the spectral bundles of the homogeneous
part (\ref{hom}) introduced in Theorem \ref{Theorem7.1.7}.

We start with the following technical lemma.

\begin{lemma}
\label{Lemma_complex}Let $\delta>0$ and $\mu\in\mathbb{C}$ with $\left\vert
\mu\right\vert \geq1$. Then there are $n_{k}\rightarrow\infty$ and $a_{n_{k}%
}\in\mathbb{C}$ with $\left\vert a_{n_{k}}\right\vert <\delta$ such that
$\mu^{n_{k}}a_{n_{k}}\in\mathbb{R}$ and $\left\vert \mu^{n_{k}}a_{n_{k}%
}\right\vert \geq\frac{\delta}{2}$ for all $k$.
\end{lemma}

\begin{proof}
With $\mu^{n}=x_{n}+\imath y_{n}$ and $a=\alpha+\imath\beta$ we have%
\[
\mu^{n}a=(x_{n}+\imath y_{n})(\alpha+\imath\beta)=x_{n}\alpha-y_{n}%
\beta+\imath(x_{n}\beta+y_{n}\alpha).
\]
If $x_{n}=0$ choose $\alpha_{n}:=0,\beta_{n}:=\frac{\delta}{2}$ to obtain
$\mu_{n}a_{n}=-y_{n}\beta_{n}\in\mathbb{R}$ and%
\[
\left\vert \mu^{n}a_{n}\right\vert =\left\vert \mu\right\vert ^{n}\left\vert
a_{n}\right\vert \geq\left\vert a_{n}\right\vert =\frac{\delta}{2}.
\]
If $x_{n}\not =0$ the product $\mu^{n}a$ is real if and only if $\beta
=-\alpha\frac{y_{n}}{x_{n}}$. According to Colonius, Cossich, and Santana
\cite[Lemma 13]{ColCS22} there are $n_{k}\rightarrow\infty$ such that
$\left\vert \frac{\operatorname{Im}(\mu^{n_{k}})}{\operatorname{Re}(\mu
^{n_{k}})}\right\vert \rightarrow0$ and hence, with $\alpha_{n_{k}}%
:=\frac{\delta}{2},\,\beta_{n_{k}}:=-\alpha_{n_{k}}\frac{y_{n_{k}}}{x_{n_{k}}%
}$, and $k$ large enough,%
\[
\left\vert \beta_{n_{k}}\right\vert =\alpha_{n_{k}}\left\vert \frac{y_{n_{k}}%
}{x_{n_{k}}}\right\vert =\frac{\delta}{2}\left\vert \frac{\operatorname{Im}%
(\mu^{n_{k}})}{\operatorname{Re}(\mu^{n_{k}})}\right\vert <\frac{\delta}{2}.
\]
It follows for $a_{n_{k}}:=\alpha_{n_{k}}+\imath\beta_{n_{k}}$ that%
\[
\left\vert a_{n_{k}}\right\vert ^{2}=\alpha_{n_{k}}^{2}+\beta_{n_{k}}%
^{2}<\frac{1}{4}\delta^{2}+\frac{1}{4}\delta^{2}\text{, and hence }\left\vert
a_{n_{k}}\right\vert <\delta.
\]
This choice of $a_{n_{k}}$ guarantees $\mu^{n_{k}}a_{n_{k}}\in\mathbb{R}$ and
using $\left\vert \mu\right\vert \geq1$%
\[
\left\vert \mu^{n_{k}}a_{n_{k}}\right\vert =\left\vert \mu\right\vert ^{n_{k}%
}\left\vert a_{n_{k}}\right\vert \geq\left\vert a_{n_{k}}\right\vert
\geq\left\vert \alpha_{n_{k}}\right\vert =\frac{\delta}{2}.
\]

\end{proof}

The next lemma relates the reachable sets and the center-unstable subspaces of
the homogeneous part.

\begin{lemma}
\label{Lemma_subbundle}Assume that the periodic linear system in
(\ref{periodic0}) without control restrictions is controllable.\textbf{ }Then
for every $\tau\in\lbrack0,T)$ the center-unstable subspace $E_{\tau}^{+,0}$
of the homogeneous part (\ref{hom}) and the reachable sets of system
(\ref{periodic0}) with controls $u\in\mathcal{U}$ satisfy%
\[
E_{\tau}^{+,0}\subset\mathbf{R}_{\mathbb{N}T+\tau}(\tau,0)\subset
\mathrm{int}\mathbf{R}_{\mathbb{N}T+\tau}(0,0).
\]

\end{lemma}

\begin{proof}
The second inclusion follows from Proposition \ref{proposition_R}. It remains
to prove the first inclusion. Since by Proposition \ref{proposition_R}
$\mathbf{R}_{\mathbb{N}T+\tau}(\tau,0)$ is convex it suffices to prove that
the real generalized eigenspaces for the eigenvalues (the Floquet multipliers)
with absolute value greater than or equal to $1$ are contained in
$\mathbf{R}_{\mathbb{N}T+\tau}(\tau,0)$. For each eigenvalue $\mu$ of
$X(T+\tau,\tau)$ and $q\in\mathbb{N}$ let $J_{q}(\mu):=\ker(\mu I-X(T+\tau
,\tau)^{q})$ and denote the set of real parts by%
\[
J_{q}^{\mathbb{R}}(\mu):=\operatorname{Re}(J_{q}(\mu))=\{\operatorname{Re}%
v\left\vert v\in J_{q}(\mu)\right.  \}.
\]
Note that $J_{q}^{\mathbb{R}}(\mu)\subset J_{q+1}^{\mathbb{R}}(\mu)$. Since
$\mathbb{C}^{d}$ splits into the direct sum of the generalized eigenspaces
$\bigcup_{q\in\{0,1,\ldots,d\}}\ker(\mu I-X(T+\tau,\tau)^{q})$ and
$X(T+\tau,\tau)$ is real it follows that $\mathbb{R}^{d}$ splits into the
direct sum of the subspaces%
\[
\bigcup\nolimits_{q\in\{0,1,\ldots,d\}}J_{q}^{\mathbb{R}}(\mu)\text{ for }%
\mu\in\mathrm{spec}(X(T+\tau,\tau)).
\]
Fix an eigenvalue $\mu=x+\imath y$ of $X(T+\tau,\tau)$ with $\left\vert
\mu\right\vert \geq1$. It suffices to show that $J_{q}^{\mathbb{R}}%
(\mu)\subset\mathbf{R}_{\mathbb{N}T+\tau}(\tau,0)$ for all $q$.

We prove the statement by induction on $q$, the case $q=0$ being trivial. So
assume that $J_{q-1}^{\mathbb{R}}(\mu)\subset\mathbf{R}_{\mathbb{N}T+\tau
}(\tau,0)$ and take any $w=w_{1}+\imath w_{2}\in J_{q}(\mu)$. We must show
that $w_{1}\in\mathbf{R}_{\mathbb{N}T+\tau}(\tau,0)$. Note that $w_{1}%
,w_{2}\in J_{q}^{\mathbb{R}}(\mu)$ (cf. Sontag \cite[p. 119]{Son98}).

For $a\in\mathbb{C}$ and $n\geq1$ one computes%
\begin{align}
X(T+\tau,\tau)^{n}aw  &  =(X(T+\tau,\tau)-\mu I+\mu I)^{n}aw\nonumber\\
&  =\sum_{j=0}^{n}{\binom{n}{j}}(X(T+\tau,\tau)-\mu I)^{n-j}\mu^{j}aw=\mu
^{n}aw+z(n), \label{X1}%
\end{align}
where $z(n):=\sum_{j=0}^{n-1}{\binom{n}{j}}(X(T+\tau,\tau)-\mu I)^{n-j}\mu
^{j}aw$. Since $aw\in J_{q}(\mu)$ it follows that $(X(T+\tau,\tau)-\mu
I)^{i}aw\in J_{q-1}(\mu)$ for all $i\geq1$, hence $z(n)\in J_{q-1}(\mu
),n\geq1$. Equality (\ref{X1}) implies
\begin{equation}
\mu^{n}aw=X(T+\tau,\tau)^{n}aw-z(n). \label{3.1}%
\end{equation}
One finds with $a=\alpha+\imath\beta$%
\[
\operatorname{Re}(aw)=\operatorname{Re}((\alpha+\imath\beta)(w_{1}+\imath
w_{2}))=\alpha w_{1}-\beta w_{2},
\]
hence%
\[
\left\Vert \operatorname{Re}(aw)\right\Vert \leq2\max(\left\vert
\alpha\right\vert ,\left\vert \beta\right\vert )\max(\left\Vert w_{1}%
\right\Vert ,\left\Vert w_{2}\right\Vert )\leq2\left\vert a\right\vert
\max(\left\Vert w_{1}\right\Vert ,\left\Vert w_{2}\right\Vert ).
\]
According to Proposition \ref{proposition_R} one has $0\in\mathrm{int}%
\mathbf{R}_{dT+\tau}(\tau,0)$. Thus there is $\delta>0$ such that
$\operatorname{Re}(aw)\in\mathbf{R}_{dT+\tau}(\tau,0)$ for all $a\in
\mathbb{C}$ with $\left\vert a\right\vert <\delta$.

By Lemma \ref{Lemma_complex} there are a sequence $(n_{k})_{k\in\mathbb{N}}$
with $n_{k}\rightarrow\infty$ and $a_{n_{k}}\in\mathbb{C}$ with $\left\vert
a_{n_{k}}\right\vert <\delta$ such that $\mu^{n_{k}}a_{n_{k}}\in\mathbb{R}$
and $\left\vert \mu^{n_{k}}a_{n_{k}}\right\vert \geq\frac{\delta}{2}$. Then it
follows that $\operatorname{Re}(a_{n_{k}}w)\in\mathbf{R}_{dT+\tau}(\tau,0)$
for all $k$.

Now choose $\ell\in\mathbb{N}$ with $\ell\geq2/\delta$. Taking real parts in
(\ref{3.1}) and choosing $a=a_{n_{k}}$ one obtains%
\begin{equation}
\mu^{n_{k}}a_{n_{k}}w_{1}=X(T+\tau,\tau)^{n_{k}}\operatorname{Re}(a_{n_{k}%
}w)-\operatorname{Re}z(n_{k}), \label{3.2}%
\end{equation}
where $\operatorname{Re}z(n_{k})\in J_{q-1}^{\mathbb{R}}(\mu)$ and
$\operatorname{Re}(a_{n_{k}}w)\in\mathbf{R}_{dT+\tau}(\tau,0)$. For $k=1$ the
variation-of-parameters formula (\ref{VdP}) with $u=0$ implies%
\[
X(T+\tau,\tau)^{n_{1}}\operatorname{Re}(a_{n_{1}}w)=X(n_{1}T+\tau
,\tau)\operatorname{Re}(a_{n_{1}}w)\in\mathbf{R}_{(n_{1}+d)T+\tau}(\tau,0).
\]
We may assume that $n_{2}\geq n_{1}+d$ and obtain%
\begin{align*}
&  X(T+\tau,\tau)^{n_{1}}\operatorname{Re}(a_{n_{1}}w)+X(T+\tau,\tau)^{n_{2}%
}\operatorname{Re}(a_{n_{2}}w)\\
&  =X(T+\tau,\tau)^{n_{1}}\left[  \operatorname{Re}(a_{n_{1}}w)+X(T+\tau
,\tau)^{d}X(T+\tau,\tau)^{n_{2}-n_{1}-d}\operatorname{Re}(a_{n_{2}}w)\right]
.
\end{align*}
With $x=\operatorname{Re}(a_{n_{1}}w)\in\mathbf{R}_{dT+\tau}(\tau,0)$ and%
\[
y=X(T+\tau,\tau)^{n_{2}-n_{1}-d}\operatorname{Re}(a_{n_{2}}w)\in
\mathbf{R}_{(n_{2}-n_{1})T+\tau}(\tau,0),
\]
Proposition \ref{Proposition9_JDDE} yields%
\[
x+X(T+\tau,\tau)^{d}y\in\mathbf{R}_{(d+n_{2}-n_{1})T+\tau}(\tau,0).
\]
Hence, using again formula (\ref{VdP}) with $u=0$ and Lemma \ref{Lemma1}(ii),%
\begin{align*}
&  X(T+\tau,\tau)^{n_{1}}\operatorname{Re}(a_{n_{1}}w)+X(T+\tau,\tau)^{n_{2}%
}\operatorname{Re}(a_{n_{2}}w)\\
&  \in X(T+\tau,\tau)^{n_{1}}\mathbf{R}_{(d+n_{2}-n_{1})T+\tau}(\tau
,0)\subset\mathbf{R}_{(d+n_{2})T+\tau}(\tau,0).
\end{align*}
In the next step we obtain for $n_{3}\geq n_{2}+d$%
\begin{align*}
&  X(T+\tau,\tau)^{n_{1}}\operatorname{Re}(a_{n_{1}}w)+X(T+\tau,\tau)^{n_{2}%
}\operatorname{Re}(a_{n_{2}}w)+X(T+\tau,\tau)^{n_{3}}\operatorname{Re}%
(a_{n_{3}}w)\\
&  \subset\mathbf{R}_{(2d+n_{3})T+\tau}(\tau,0).
\end{align*}
Proceeding in this way, we arrive at%
\[
\sum_{k=1}^{\ell}X(T+\tau,\tau)^{n_{k}}\operatorname{Re}(a_{n_{k}}%
w)\in\mathbf{R}_{((\ell-1)d+n_{\ell})T+\tau}(\tau,0).
\]
Summing (\ref{3.2}) from $k=1$ to $\ell$ this yields%
\begin{align*}
\sum_{k=1}^{\ell}\mu^{n_{k}}a_{n_{k}}w_{1}  &  =\sum_{k=1}^{\ell}\left[
X(T+\tau,\tau)^{n_{k}}\operatorname{Re}(a_{n_{k}}w)-\operatorname{Re}%
z(n_{k})\right] \\
&  \in\mathbf{R}_{((\ell-1)d+n_{\ell})T+\tau}(\tau,0)+J_{q-1}^{\mathbb{R}}%
(\mu)\subset\mathbf{R}_{\mathbb{N}T+\tau}(\tau,0)+J_{q-1}^{\mathbb{R}}(\mu).
\end{align*}
By the induction hypothesis the linear subspace $J_{q-1}^{\mathbb{R}}(\mu)$ is
contained in the convex set $\mathbf{R}_{\mathbb{N}T+\tau}(\tau,0)$, which is
open by Proposition \ref{proposition_R}. This implies (cf. Sontag \cite[Lemma
3.6.4]{Son98}) that $\mathbf{R}_{\mathbb{N}T+\tau}(\tau,0)+J_{q-1}%
^{\mathbb{R}}(\mu)=\mathbf{R}_{\mathbb{N}T+\tau}(\tau,0)$. If $\mu^{n_{k}%
}a_{n_{k}}>0$ for all $k\in\{1,\dotsc,\ell\}$, then the real number
$\rho:=\sum_{k=1}^{\ell}\mu^{n_{k}}a_{n_{k}}>\ell\cdot\delta/2\geq1$. For the
$k$ with $\mu^{n_{k}}a_{n_{k}}<0$, replace $a_{n_{k}}$ by $-a_{n_{k}}$ to get
the same conclusion. It follows that $w_{1}$ is a convex combination of the
points $0$ and $\rho w_{1}$ in the convex set $\mathbf{R}_{\mathbb{N}T+\tau
}(\tau,0)$:%
\[
w_{1}=\left(  1-\rho^{-1}\right)  \cdot0+\rho^{-1}\cdot\rho w_{1}.
\]
It follows that $w_{1}\in\mathbf{R}_{\mathbb{N}T+\tau}(\tau,0)$ completing the
induction step. We have shown that $E_{\tau}^{+,0}\subset\mathbf{R}%
_{\mathbb{N}T+\tau}(\tau,0)$ proving the lemma.
\end{proof}

The following result characterizes the reachable and controllable sets of the
autonomized system (\ref{aug1}) by spectral properties of the homogeneous part
(\ref{hom}). Recall that we denote the spectral subbundles of the $T$-periodic
linear differential equation (\ref{hom}) by $\mathcal{E}^{-},\mathcal{E}%
^{+},\mathcal{E}^{+,0},$ and $\mathcal{E}^{-,0}$. For subsets $K_{\tau}%
\subset\mathbb{R}^{d}$ and matrices $Y(\tau),\tau\in\mathbb{S}^{1}=[0,T)$ we
use the following notation:%
\[
\mathcal{K}:=\bigcup_{\tau\in\lbrack0,T)}\left\{  (\tau,x)\in\mathbb{S}%
^{1}\times\mathbb{R}^{d}\left\vert x\in K_{\tau}\right.  \right\}
,\,Y(\cdot)\mathcal{K}:=\bigcup_{\tau\in\lbrack0,T)}\left\{  (\tau
,x)\left\vert x\in Y(\tau)K_{\tau}\right.  \right\}  .
\]

\begin{theorem}
\label{Theorem_sub}Suppose that the periodic system in (\ref{periodic0}) with
unconstrained controls is controllable and consider the autonomized system
(\ref{aug1}) with controls $u\in\mathcal{U}$.

(i) Then the reachable set $\mathbf{R}^{a}(0,0)\subset\mathbb{S}^{1}%
\times\mathbb{R}^{d}$ satisfies $\mathbb{S}^{1}\times\{0\}\subset
\mathrm{int}\mathbf{R}^{a}(0,0)$ and%
\[
X(dT+\cdot,\cdot)\mathcal{K}^{-}\oplus\mathcal{E}^{+,0}\subset\mathrm{int}%
\mathbf{R}^{a}(0,0)\subset\mathcal{K}^{-}\oplus\mathcal{E}^{+,0},
\]
with uniformly bounded convex sets $K_{\tau}^{-}:=\mathrm{int}\mathbf{R}%
_{\mathbb{N}T+\tau}(0,0)\cap E_{\tau}^{-},\tau\in\mathbb{S}^{1}$.

(ii)\ The controllable set $\mathbf{C}^{a}(0,0)$ satisfies $\mathbb{S}%
^{1}\times\{0\}\subset\mathrm{int}\mathbf{C}^{a}(0,0)$ and%
\[
\mathcal{E}^{-,0}\oplus X(-dT+\cdot,\cdot)\mathcal{K}^{+}\subset
\mathrm{int}\mathbf{C}^{a}(0,0)\subset\mathcal{E}^{-,0}\oplus\mathcal{K}^{+}%
\]
with uniformly bounded convex sets $K_{\tau}^{+}:=\mathrm{int}\mathbf{C}%
_{-\mathbb{N}T+\tau}(0,0)\cap E_{\tau}^{+},\tau\in\mathbb{S}^{1}$.
\end{theorem}

\begin{proof}
(i) Lemma \ref{Lemma3.3} and Lemma \ref{Lemma3.4} imply that $\mathbb{S}%
^{1}\times\{0\}\subset\mathrm{int}\mathbf{R}^{a}(0,0)$ and%
\begin{equation}
\mathbf{R}^{a}(0,0)=\bigcup\nolimits_{\tau\in\lbrack0,T)}\left\{  (\tau
,x)\in\mathbb{S}^{1}\times\mathbb{R}^{d}\left\vert x\in\mathbf{R}%
_{\mathbb{N}T+\tau}(0,0)\right.  \right\}  . \label{R_a3}%
\end{equation}
We claim that%
\begin{equation}
\mathrm{int}\mathbf{R}_{\mathbb{N}T+\tau}(0,0)=K_{\tau}^{-}+E_{\tau}%
^{+,0}\text{ for every }\tau\in\lbrack0,T). \label{4.6}%
\end{equation}
Recall that by Proposition \ref{proposition_R} the set $\mathrm{int}%
\mathbf{R}_{\mathbb{N}T+\tau}(0,0)$ is convex. Lemma \ref{Lemma_subbundle}
shows that $E_{\tau}^{+,0}\subset\mathbf{R}_{\mathbb{N}T+\tau}(\tau
,0)\subset\mathrm{int}\mathbf{R}_{\mathbb{N}T+\tau}(0,0)$. By Sontag
\cite[Lemma 3.6.4]{Son98} it follows that%
\[
K_{\tau}^{-}+E_{\tau}^{+,0}\subset\mathrm{int}\mathbf{R}_{\mathbb{N}T+\tau
}(0,0)+E_{\tau}^{+,0}=\mathrm{int}\mathbf{R}_{\mathbb{N}T+\tau}(0,0).
\]
For the converse inclusion, write $x\in\mathrm{int}\mathbf{R}_{\mathbb{N}%
T+\tau}(0,0)$ as $x=y\oplus z$ with $y\in E_{\tau}^{-}$ and $z\in E_{\tau
}^{+,0}$. Again by \cite[Lemma 3.6.4]{Son98} it follows that%
\[
y=x-z\in\mathrm{int}\mathbf{R}_{\mathbb{N}T+\tau}(0,0)+E_{\tau}^{+,0}%
=\mathrm{int}\mathbf{R}_{\mathbb{N}T+\tau}(0,0),
\]
which proves that $y\in K_{\tau}^{-}$ and therefore $x\in K_{\tau}^{-}%
+E_{\tau}^{+,0}$. This proves (\ref{4.6}).

By (\ref{R_a3}) it follows that $\mathrm{int}\mathbf{R}^{a}(0,0)\subset
\mathcal{K}^{-}\oplus\mathcal{E}^{+,0}$ proving the second inclusion in (i).

For the first inclusion in (i) consider $x=\varphi(kT+\tau;0,0,u)\in
\mathbf{R}_{\mathbb{N}T+\tau}(0,0)$ and recall that by Proposition
\ref{proposition_R} there is a ball $\mathbf{B}(0;\varepsilon)\subset
\mathbf{R}_{dT+\tau}(\tau,0)$ for all $\tau\in\mathbb{S}^{1}$. For
$y=\varphi(dT+\tau;\tau,0,v)\in\mathbf{B}(0;\varepsilon)$ define%
\[
w(t)=\left\{
\begin{array}
[c]{lll}%
u(t) & \text{for} & t\in\lbrack0,kT+\tau)\\
v(t-kT) & \text{for} & t\in\lbrack kT+\tau,(d+k)T+\tau]
\end{array}
\right.  .
\]
This implies%
\begin{align*}
&  \varphi((k+d)T+\tau;0,0,w)=\varphi((k+d)T+\tau;kT+\tau,x,w)\\
&  =X((k+d)T+\tau,kT+\tau)x+\int_{kT+\tau}^{(k+d)T+\tau}X((k+d)T+\tau
,s)B(s)w(s)ds\\
&  =X(dT+\tau,\tau)x+\int_{\tau}^{dT+\tau}X(dT+\tau,s)B(s)v(s)ds=X(dT+\tau
,\tau)x+y,
\end{align*}
showing that%
\[
X((dT+\tau,\tau)x+\mathbf{B}(0;\varepsilon)\subset\mathbf{R}_{\mathbb{N}%
T+\tau}(0,0),\tau\in\lbrack0,T).
\]
It follows that%
\[
\left\{  (\tau,x)\in\mathbb{S}^{1}\times\mathbb{R}^{d}\left\vert x\in
X(dT+\tau,\tau)\mathbf{R}_{\mathbb{N}T+\tau}(0,0)\right.  \right\}
\subset\mathrm{int}\mathbf{R}^{a}(0,0),\tau\in\mathbb{S}^{1}.
\]
Since $X(dT+\tau,\tau)E_{\tau}^{+,0}=E_{\tau}^{+,0}$ equality (\ref{4.6})
implies%
\[
X(dT+\tau,\tau)\mathrm{int}\mathbf{R}_{\mathbb{N}T+\tau}(0,0)=X(dT+\tau
,\tau)(K_{\tau}^{-}+E_{\tau}^{+,0})=X(dT+\tau,\tau)K_{\tau}^{-}+E_{\tau}^{+,0}%
\]
and $X(dT+\tau,\tau)K_{\tau}^{-}\subset E_{\tau}^{-}$. This shows the first
inclusion in assertion (i),%
\begin{align*}
X(dT+\tau,\tau)\mathcal{K}^{-}\oplus\mathcal{E}^{+,0}  &  =\bigcup
\nolimits_{\tau\in\lbrack0,T)}\left\{  (\tau,x)\in\mathbb{S}^{1}%
\times\mathbb{R}^{d}\left\vert x\in X(dT+\tau,\tau)K_{\tau}^{-}+E_{\tau}%
^{+,0}\right.  \right\} \\
&  \subset\mathrm{int}\mathbf{R}^{a}(0,0).
\end{align*}
In order to prove that $K_{\tau}^{-}$ is bounded, let $x=\varphi
(kT+\tau;0,0,u)\in\mathbf{R}_{\mathbb{N}T+\tau}(0,0)\cap E_{\tau}^{-}$. Then
using linearity%
\begin{align}
x  &  =\varphi(kT+\tau;0,0,u)=\varphi(kT+\tau;\tau,\varphi(\tau
,0,0,u),u)\label{B0}\\
&  =X(kT+\tau,\tau)\varphi(\tau;0,0,u)+\varphi(kT+\tau;\tau,0,u).\nonumber
\end{align}
Define for $\tau\in\mathbb{S}^{1}$ a bounded linear map by%
\[
\mathcal{B}_{\tau}:L^{\infty}([0,T],\mathbb{R}^{m})\rightarrow\mathbb{R}%
^{d},\,\mathcal{B}_{\tau}(u^{\prime})=\int_{0}^{T}X(T+\tau,\tau+s)B(\tau
+s)u^{\prime}(s)ds.
\]
Using the variation-of-parameters formula (\ref{VdP}) and periodicity one
computes%
\begin{align*}
&  \varphi(kT+\tau;\tau,0,u)=\int_{\tau}^{kT+\tau}X(kT+\tau,s)B(s)u(s)ds\\
&  =\sum_{j=0}^{k-1}\int_{jT+\tau}^{(j+1)T+\tau}X(kT+\tau,s)B(s)u(s)ds\\
&  =\sum_{j=0}^{k-1}X(T+\tau,\tau)^{k-j}\int_{0}^{T}X((j+1)T+\tau
,jT+\tau+s)B(jT+\tau+s)u(jT+\tau+s)ds\\
&  =\sum_{j=0}^{k-1}X(T+\tau,\tau)^{k-j}\int_{0}^{T}X(T+\tau,\tau
+s)B(\tau+s)u(jT+\tau+s)ds\\
&  =\sum_{j=0}^{k-1}X(T+\tau,\tau)^{k-j}\mathcal{B}_{\tau}(u(jT+\tau+\cdot)).
\end{align*}
Consider the projection $\pi_{\tau}:\mathbb{R}^{d}=E_{\tau}^{-}\oplus E_{\tau
}^{+,0}\rightarrow$ $E_{\tau}^{-}$ along $E_{\tau}^{+,0}$. By Theorem
\ref{Theorem7.1.7}(ii) the subspaces $E_{\tau}^{-}$ and $E_{\tau}^{+,0}$ are
$X(T+\tau,\tau)$-invariant, hence $\pi_{\tau}$ commutes with $X(T+\tau,\tau)$.
With (\ref{B0}) this yields%
\begin{align*}
x  &  =\pi_{\tau}x=\pi_{\tau}X(kT+\tau,\tau)\varphi(\tau;0,0,u)+\pi_{\tau
}\varphi(kT+\tau;\tau,0,u)\\
&  =X(T+\tau,\tau)^{k}\pi_{\tau}\varphi(\tau;0,0,u)+\sum_{j=0}^{k-1}%
X(T+\tau,\tau)^{k-j}\pi_{\tau}\mathcal{B}_{\tau}(u(jT+\tau+\cdot)).
\end{align*}
Since $X(T+\tau,\tau)|_{E_{\tau}^{-}}$ is a linear contraction there exist
constants $a\in(0,1)$ and $c\geq1$ such that $\Vert X(T+\tau,\tau
)^{n}x^{\prime}\Vert\leq ca^{n}\Vert x^{\prime}\Vert$ for all $n\in\mathbb{N}$
and $x^{\prime}\in E_{\tau}^{-}$. These constants may be chosen independently
of $\tau\in\mathbb{S}^{1}$. It follows that%
\[
\left\Vert X(T+\tau,\tau)^{k}\pi_{\tau}\varphi(\tau;0,0,u)\right\Vert \leq
ca^{k}\left\Vert \pi_{\tau}\varphi(\tau;0,0,u)\right\Vert
\]
and%
\[
\left\Vert {}\right.  \sum_{j=0}^{k-1}X(T+\tau,\tau)^{k-j}\pi_{\tau
}\mathcal{B}_{\tau}(\left.  u(jT+\tau+\cdot))\right\Vert \leq\sum_{j=0}%
^{k-1}ca^{k-j}\left\Vert \pi_{\tau}\mathcal{B}_{\tau}(u(jT+\tau+\cdot
))\right\Vert .
\]
Since $U$ is compact, there is $M>0$ such that $\left\Vert \pi_{\tau}%
\varphi(\tau,0,0,u^{\prime})\right\Vert ,\Vert\pi_{\tau}\mathcal{B}_{\tau
}(u^{\prime})\Vert\leq M$ for all $\tau\in\mathbb{S}^{1}$ and $u^{\prime}%
\in\mathcal{U}$. Thus $K_{\tau}^{-}$ is bounded by
\[
\left\Vert x\right\Vert =\left\Vert \varphi(kT+\tau;0,0,u)\right\Vert \leq
ca^{k}M+cM\sum_{j=0}^{k-1}a^{k-j}\leq\dfrac{2cM}{1-a}.
\]
Assertion (ii) follows by considering the time-reversed systems.
\end{proof}

Next we define subsets of complete approximate controllability.

\begin{definition}
\label{Definition_control_sets}A nonvoid set $D^{a}\subset$ $\mathbb{S}%
^{1}\times\mathbb{R}^{d}$ is a control set of the autonomized system
(\ref{aug1}) on $\mathbb{S}^{1}\times\mathbb{R}^{d}$ if it has the following
properties: (i) for all $(\tau,x)\in D^{a}$ there is a control $u\in
\mathcal{U}$ such that $\varphi^{a}(t,(\tau,x),u)\in D^{a}$ for all $t\geq0$,
(ii) for all $(\tau,x)\in D^{a}$ one has $D^{a}\subset\overline{\mathbf{R}%
^{a}(\tau,x)}$, and (iii) $D^{a}$ is maximal with these properties, that is,
if $D^{\prime}\supset D^{a}$ satisfies conditions (i) and (ii), then
$D^{\prime}=D^{a}$.
\end{definition}

The following lemma shows that there is a control set around $(0,0)$.

\begin{lemma}
\label{Lemma_point}Suppose that the periodic system in (\ref{periodic0}) with
unconstrained controls is controllable. Then $D^{a}:=\overline{\mathbf{R}%
^{a}(0,0)}\cap\mathbf{C}^{a}(0,0)$ is a control set and $\mathbb{S}^{1}%
\times\{0\}\subset\mathrm{int}D^{a}$.
\end{lemma}

\begin{proof}
Theorem \ref{Theorem_sub} shows that $\mathbb{S}^{1}\times\{0\}\subset
\mathrm{int}\mathbf{R}^{a}(0,0)\cap\mathbf{C}^{a}(0,0)$. Consider
$(\tau,x),(\sigma,y)\in\overline{\mathbf{R}^{a}(0,0)}\cap\mathbf{C}^{a}(0,0)$
and let $\varepsilon>0$. Then there are $t_{1},t_{2}\geq0$ and $u_{1},u_{2}%
\in\mathcal{U}$ with
\[
\varphi^{a}(t_{1};(\sigma,y),u_{1})=(0,0)\text{ and }d(\varphi^{a}%
(t_{2};(0,0),u_{2}),(\tau,x))<\varepsilon.
\]
Define a control $v$ by%
\[
v(t):=\left\{
\begin{array}
[c]{lll}%
u_{1}(t) & \text{for} & t\in\lbrack0,t_{1}]\\
u_{2}(t-t_{1}) & \text{for} & t\in(t_{1},t_{1}+t_{2}]
\end{array}
\right.  .
\]
Then it follows that $d(\varphi^{a}(t_{1}+t_{2},(\sigma,y),v),(\tau
,x))<\varepsilon$. This shows that%
\begin{equation}
\overline{\mathbf{R}^{a}(0,0)}\cap\mathbf{C}^{a}(0,0)\subset\overline
{\mathbf{R}^{a}(\tau,x)}\text{ for all }(\tau,x)\in\overline{\mathbf{R}%
^{a}(0,0)}\cap\mathbf{C}^{a}(0,0). \label{4.9}%
\end{equation}
Define $D^{a}$ as the union of all sets $D^{\prime}$ with $D^{\prime}%
\subset\overline{\mathbf{R}^{a}(\tau,x)}$ for all $(\tau,x)\in D^{\prime}$ and
$\overline{\mathbf{R}^{a}(0,0)}\cap\mathbf{C}^{a}(0,0)\subset D^{\prime}$.
Then any $(\tau,x)\in D^{a}$ is in some set $D^{\prime}$ and $(0,0)\in
\overline{\mathbf{R}^{a}(0,0)}\cap\mathrm{int}\mathbf{C}^{a}(0,0)$ implies
that there are $t>0$ and $u\in\mathcal{U}$ with $\varphi^{a}(t;(\tau
,x),u)\in\mathbf{C}^{a}(0,0)$ and $(\tau,x)\in\mathbf{C}^{a}(0,0)$ follows.
This shows that $D^{a}\subset\mathbf{C}^{a}(0,0)$. Since $D^{a}\subset
\overline{\mathbf{R}^{a}(0,0)}$ this also implies $D^{a}\subset\overline
{\mathbf{R}^{a}(\tau,x)}$ proving $D^{a}=\overline{\mathbf{R}^{a}(0,0)}%
\cap\mathbf{C}^{a}(0,0)$. It also follows that $D^{a}$ is a maximal set with
$D^{a}\subset\overline{\mathbf{R}^{a}(\tau,x)}$ and $\mathrm{int}D^{a}%
\not =\varnothing$. Hence Kawan \cite[Proposition 1.20]{Kawa13} yields that
$D^{a}$ is a control set.
\end{proof}

The following theorem characterizes the unique control set with nonvoid
interior of the autonomized system. Recall that the center subbundle
$\mathcal{E}^{0}$ of the periodic linear differential equation (\ref{hom}) is
non-trivial if and only if $0$ is a Floquet exponent if and only if there is a
Floquet multiplier of modulus $1$, i.e., if $\mathrm{spec}(X(T,0))\cap
\mathbb{S}^{1}\not =\varnothing$.

\begin{theorem}
\label{Theorem_cs1}Suppose that the periodic system in (\ref{periodic0}) with
unconstrained controls is controllable. Then there exists a unique control set
$D^{a}$ with nonvoid interior of the autonomized system (\ref{aug1}) with
controls $u\in\mathcal{U}$.

It is given by $D^{a}=\overline{\mathbf{R}^{a}(0,0)}\cap\mathbf{C}^{a}(0,0)$
and satisfies $\mathbb{S}^{1}\times\{0\}\subset\mathrm{int}D^{a}$ and, with
$\mathcal{K}^{-}\subset\mathcal{E}^{-}$ and $\mathcal{K}^{+}\subset
\mathcal{E}^{+}$ defined in Theorem \ref{Theorem_sub} and $Y(\tau
):=X(dT+\tau,\tau),\tau\in\mathbb{S}^{1}$,%
\begin{equation}
Y(\cdot)\mathcal{K}^{-}\oplus\mathcal{E}^{0}\oplus Y(\cdot)^{-1}%
\mathcal{K}^{+}\subset\mathrm{int}D^{a}\subset D^{a}\subset Y(\cdot
)^{-1}\overline{\mathcal{K}^{-}}\oplus\mathcal{E}^{0}\oplus Y(\cdot
)\mathcal{K}^{+}. \label{D_a2}%
\end{equation}
In particular, $\mathrm{int}D^{a}$ is unbounded if and only if the center
subbundle $\mathcal{E}^{0}$ is nontrivial.
\end{theorem}

\begin{proof}
The inclusion $\mathbb{S}^{1}\times\{0\}\subset\mathrm{int}D^{a}$ and
$D^{a}=\overline{\mathbf{R}^{a}(0,0)}\cap\mathbf{C}^{a}(0,0)$ follow by Lemma
\ref{Lemma_point}. Furthermore, the inclusions (\ref{D_a2}) imply the last
assertion since $X(dT+\tau,\tau),\tau\in\mathbb{S}^{1}$, as well as
$\mathcal{K}^{-}$ and $\mathcal{K}^{+}$ are bounded. Theorem \ref{Theorem_sub}
implies%
\[
X(dT+\cdot,\cdot)\mathcal{K}^{-}\oplus\mathcal{E}^{+}\oplus\mathcal{E}%
^{0}\subset\mathrm{int}\mathbf{R}^{a}(0,0),\,\mathcal{E}^{-}\oplus
\mathcal{E}^{0}\oplus X(-dT+\cdot,\cdot)\mathcal{K}^{+}\subset\mathrm{int}%
\mathbf{C}^{a}(0,0).
\]
Since $\mathrm{int}D^{a}\subset\mathrm{int}\mathbf{R}^{a}(0,0)\cap
\mathrm{int}\mathbf{C}^{a}(0,0)$ and $X(-dT+\tau,\tau)=X(\tau,-dT+\tau
)^{-1}=X(dT+\tau,\tau)^{-1}$ for $\tau\in\mathbb{S}^{1}$ the first inclusion
in (\ref{D_a2}) follows. In order to prove the third inclusion let
$(\tau,y)\in\mathbf{R}^{a}(0,0)$ be given by%
\[
(\tau,y)=\varphi^{a}(t;(0,0),u)=(t\operatorname{mod}T,\varphi(t;0,0,u)=(\tau
,\varphi(\ell T+\tau;0,u))
\]
with $y=\varphi(\ell T+\tau;0,u)\in\mathbf{R}_{\mathbb{N}T+\tau}(0,0)$. By
Proposition \ref{proposition_R} there is a ball $\mathbf{B}(0;\varepsilon
)\subset\mathbf{R}_{dT+\tau}(\tau,0)$. Hence Proposition
\ref{Proposition9_JDDE} implies%
\[
\mathbf{B}(0;\varepsilon)+X(dT+\tau,\tau)y\subset\mathbf{R}_{(d+\ell)T+\tau
}(\tau,0)\subset\mathbf{R}_{\mathbb{N}T+\tau}(\tau,0).
\]
Since $\varepsilon>0$ is independent of $\tau\in\mathbb{S}^{1}$ it follows
that $X(dT+\cdot,\cdot)\mathbf{R}^{a}(0,0)\subset\mathrm{int}\mathbf{R}%
^{a}(0,0)$, and hence%
\[
X(dT+\cdot,\cdot)\overline{\mathbf{R}^{a}(0,0)}\subset\overline{\mathrm{int}%
\mathbf{R}^{a}(0,0)}.
\]
Analogously it follows that%
\[
X(-dT+\cdot,\cdot)\mathbf{C}^{a}(0,0)\subset\mathrm{int}\mathbf{C}^{a}(0,0).
\]
By Theorem \ref{Theorem_sub} we obtain for $x\in D^{a}=\overline
{\mathbf{R}^{a}(0,0)}\cap\mathbf{C}^{a}(0,0)$,%
\begin{align*}
X(dT+\cdot,\cdot)x  &  \in\overline{\mathrm{int}\mathbf{R}^{a}(0,0)}%
\subset\overline{\mathcal{K}^{-}\oplus\mathcal{E}^{+,0}}=\overline
{\mathcal{K}^{-}}\oplus\mathcal{E}^{0}\oplus\mathcal{E}^{+},\\
X(dT+\cdot,\cdot)^{-1}x  &  =X(-dT+\cdot,\cdot)x\in\mathrm{int}\mathbf{C}%
^{a}(0,0)\subset\mathcal{E}^{-}\oplus\mathcal{E}^{0}\oplus\mathcal{K}^{+}.
\end{align*}
This implies%
\[
x\in X(dT+\cdot,\cdot)^{-1}\left(  \overline{\mathcal{K}^{-}}\oplus
\mathcal{E}^{0}\oplus\mathcal{E}^{+}\right)  \cap X(dT+\cdot,\cdot)\left(
\mathcal{E}^{-}\oplus\mathcal{E}^{0}\oplus\mathcal{K}^{+}\right)  .
\]
By Theorem \ref{Theorem7.1.7} the subbundles $\mathcal{E}^{0}$ and
$\mathcal{E}^{\pm}$ are invariant under $X(dT+\cdot,\cdot)$ and hence%
\[
D^{a}\subset X(dT+\cdot,\cdot)^{-1}\overline{\mathcal{K}^{-}}\oplus
\mathcal{E}^{0}\oplus X(dT+\cdot,\cdot)\mathcal{K}^{+}=Y(\cdot)^{-1}%
\overline{\mathcal{K}^{-}}\oplus\mathcal{E}^{0}\oplus Y(\cdot)\mathcal{K}%
^{+}.
\]
This proves the third inclusion in (\ref{D_a2}).

It remains to show uniqueness. Let $E\subset\mathbb{S}^{1}\times\mathbb{R}%
^{d}$ be an arbitrary control set with nonvoid interior. For all $(\tau,x)\in
E$ there is a control function $u\in\mathcal{U}$ such that $\varphi
^{a}(t;(\tau,x),u)\in E$ for all $t\geq0$, hence $E\cap\left(  \{0\}\times
\mathbb{R}^{d}\right)  \not =\varnothing$.

By linearity, it follows from $\varphi^{a}(t;(0,x_{1}),u)=(\tau,x_{2})$ for
$x_{1},x_{2}\in\mathbb{R}^{d}$ and $t>0$ that $\varphi^{a}(t;\alpha
x_{1},\alpha u)=(\tau,\alpha x_{2})$ for any $\alpha\in(0,1]$. This implies
that $\{(0,\alpha x)\left\vert (0,x)\in E\right.  \}$ is contained in some
control set $D_{\alpha}$ and $\mathrm{int}\{(0,\alpha x)\left\vert (0,x)\in
E\right.  \}\subset\mathrm{int}D_{\alpha}$. Now choose any $(0,x)\in
\mathrm{int}E$ and suppose, by way of contradiction, that
\[
\alpha_{0}:=\inf\{\alpha\in(0,1]\left\vert \forall\beta\in\lbrack
\alpha,1]:(0,\beta x)\in E\right.  \}>0.
\]
Then $(0,\alpha_{0}x)\in\partial E$ and $(0,\alpha_{0}x)\in\mathrm{int}%
D_{\alpha_{0}}$. Therefore $E\cap\mathrm{int}D_{\alpha_{0}}\not =\varnothing$,
and it follows that $E=D_{\alpha_{0}}$ and $(0,\alpha_{0}x)\in\mathrm{int}E$.
This is a contradiction and so $\alpha_{0}=0$. Choosing $\alpha>0$ small
enough such that $(0,\alpha x)\in D^{a}$, we obtain $(0,\alpha x)\in E\cap
D^{a}$ and it follows that $E=D^{a}$.
\end{proof}

\begin{remark}
A control system is called locally accessible, if the reachable and
controllable sets up to time $t>0$ have nonvoid interior for every $t>0$. If
this holds for the autonomized system (\ref{aug1}), then Colonius and Kliemann
\cite[Lemma 3.2.13(i)]{ColK00} implies that $\overline{D^{a}}=\overline
{\mathrm{int}D^{a}}$. Even if the system in (\ref{periodic0}) without control
restrictions is controllable, the autonomized system (\ref{aug1}) need not
satisfy $\mathrm{int}\mathbf{R}_{t}^{a}(\tau,x)\not =\varnothing$ for small
$t>0$, hence, in general, it is not locally accessible. The example in
Bittanti, Guarbadassi, Mafezzoni, and Silverman \cite[p. 38]{Bit78} is a counterexample.
\end{remark}

\begin{remark}
\label{RemarkGayer}Gayer \cite[Theorem 3]{Gayer05} relates the control sets of
autonomized (general nonlinear) control systems to control sets of
discrete-time systems depending on $\tau\in\mathbb{S}^{1}$defined by
Poincar\'{e} maps. For system (\ref{aug1}) these systems are defined by%
\[
\Phi_{\tau}^{u}:\mathbb{R}^{d}\rightarrow\mathbb{R}^{d},\,\Phi_{\tau}%
^{u}(\cdot)=\varphi^{a}(T,(\tau,\cdot),u)=\varphi(T+\tau;\tau,\cdot,u)\quad
u\in\mathcal{U}.
\]

\end{remark}

\section{The Poincar\'{e} sphere\label{Section5}}

This section describes the global controllability behavior of periodic linear
control systems of the form (\ref{periodic0}) with homogeneous part
(\ref{hom}) by projection to the Poincar\'{e} sphere. This allows us to
determine the behavior \textquotedblleft near infinity\textquotedblright\ by
the induced system near the equator.

The system on the Poincar\'{e} sphere is obtained by attaching the state space
$\mathbb{R}^{d}$ to the north pole $(0,1)\in\mathbb{R}^{d}\times\mathbb{R}$ of
the unit sphere $\mathbb{S}^{d}$ in $\mathbb{R}^{d+1}$ and then taking the
stereographic projection to $\mathbb{S}^{d}$. More formally, the extended
system with scalar part $\dot{z}=0$ is defined as%
\begin{equation}
\left(
\begin{array}
[c]{c}%
\dot{x}(t)\\
\dot{z}(t)
\end{array}
\right)  =\left(
\begin{array}
[c]{cc}%
A(t) & 0\\
0 & 0
\end{array}
\right)  \left(
\begin{array}
[c]{c}%
x(t)\\
z(t)
\end{array}
\right)  +\sum_{i=1}^{m}u_{i}(t)\left(
\begin{array}
[c]{cc}%
0 & b_{i}(t)\\
0 & 0
\end{array}
\right)  \left(
\begin{array}
[c]{c}%
x(t)\\
z(t)
\end{array}
\right)  , \label{ext}%
\end{equation}
where $b_{i}(t)$ denote the columns of $B(t)$. For $z\equiv1$ we get a copy of
the original system (\ref{periodic0}). Abbreviate%
\[
\hat{A}(t)=\left(
\begin{array}
[c]{cc}%
A(t) & 0\\
0 & 0
\end{array}
\right)  ,~\hat{B}_{i}(t)=\left(
\begin{array}
[c]{cc}%
0 & b_{i}(t)\\
0 & 0
\end{array}
\right)  ,~\sum_{i=1}^{m}u_{i}(t)\hat{B}_{i}(t)=\left(
\begin{array}
[c]{cc}%
0 & B(t)u(t)\\
0 & 0
\end{array}
\right)  .
\]
The projection of the homogeneous control system (\ref{ext}) on $\mathbb{S}%
^{d}\subset\mathbb{R}^{d+1}$ has the form (omitting the argument $t$)%
\begin{align*}
\left(
\begin{array}
[c]{c}%
\dot{s}\\
\dot{s}_{d+1}%
\end{array}
\right)   &  =[\hat{A}-(s^{\top},s_{d+1})\hat{A}(s^{\top},s_{d+1})^{\top}\cdot
I_{d+1}](s^{\top},s_{d+1})^{\top}\\
&  \qquad+\sum_{i=1}^{m}u_{i}[\hat{B}_{i}-(s^{\top},s_{d+1})\hat{B}%
_{i}(s^{\top},s_{d+1})^{\top}\cdot I_{d+1}](s^{\top},s_{d+1})^{\top}.
\end{align*}
This is obtained by subtracting the radial components of the linear vector
fields $\hat{A}(t)$ and $\hat{B}_{i}(t)$.

We compute%
\begin{align}
&  \left(
\begin{array}
[c]{c}%
\dot{s}\\
\dot{s}_{d+1}%
\end{array}
\right)  =\left[  \left(
\begin{array}
[c]{cc}%
A & 0\\
0 & 0
\end{array}
\right)  -(s^{\top},s_{d+1})\left(
\begin{array}
[c]{cc}%
A & 0\\
0 & 0
\end{array}
\right)  \left(
\begin{array}
[c]{c}%
s\\
s_{d+1}%
\end{array}
\right)  \cdot I_{d+1}\right]  \left(
\begin{array}
[c]{c}%
s\\
s_{d+1}%
\end{array}
\right) \nonumber\\
&  \qquad+\sum_{i=1}^{m}u_{i}\left[  \left(
\begin{array}
[c]{cc}%
0 & b_{i}\\
0 & 0
\end{array}
\right)  -(s^{\top},s_{d+1})\left(
\begin{array}
[c]{cc}%
0 & b_{i}\\
0 & 0
\end{array}
\right)  \left(
\begin{array}
[c]{c}%
s\\
s_{d+1}%
\end{array}
\right)  \cdot I_{d+1}\right]  \left(
\begin{array}
[c]{c}%
s\\
s_{d+1}%
\end{array}
\right) \nonumber\\
\medskip &  =\left(
\begin{array}
[c]{cc}%
A-s^{\top}As\cdot I_{d} & 0\\
0 & -s^{\top}As
\end{array}
\right)  \left(
\begin{array}
[c]{c}%
s\\
s_{d+1}%
\end{array}
\right) \nonumber\\
&  \qquad+\sum_{i=1}^{m}u_{i}\left(
\begin{array}
[c]{cc}%
-s^{\top}b_{i}s_{d+1}\cdot I_{d} & b_{i}\\
0 & -s^{\top}b_{i}s_{d+1}%
\end{array}
\right)  \left(
\begin{array}
[c]{c}%
s\\
s_{d+1}%
\end{array}
\right) \nonumber\\
&  =\left(
\begin{array}
[c]{c}%
\left[  A-s^{\top}As\cdot I_{d}\right]  s\\
-s^{\top}As~s_{d+1}%
\end{array}
\right)  +\sum_{i=1}^{m}u_{i}\left(
\begin{array}
[c]{c}%
-s^{\top}b_{i}s_{d+1}s+b_{i}s_{d+1}\\
-s^{\top}b_{i}s_{d+1}^{2}%
\end{array}
\right)  . \label{5.3}%
\end{align}
This is the system equation for the induced control system on the Poincar\'{e}
sphere. By adding the phase $\tau\in\mathbb{S}^{1}$ this induces an autonomous
control system on $\mathbb{S}^{1}\times\mathbb{S}^{d}$.

\begin{remark}
\label{Remark_projective}The homogeneous control system (\ref{5.3}) also
induces a control system on projective space $\mathbb{P}^{d}$ and a
corresponding autonomized control system on $\mathbb{S}^{1}\times
\mathbb{P}^{d}$. Parallel to the following developments on the unit sphere
$\mathbb{S}^{d}$ one may also work with $\mathbb{P}^{d}$. Here we prefer to
work on the sphere since this allows us to write down everything explicitly.
\end{remark}

On the \textquotedblleft equator\textquotedblright\ of the sphere
$\mathbb{S}^{d}$ given by%
\[
\mathbb{S}^{d,0}:=\{s=(s_{1},\ldots,s_{d},s_{d+1})\in\mathbb{S}^{d}\left\vert
s_{d+1}=0\right.  \},
\]
the first $d$ components of (\ref{5.3}) reduce to the (uncontrolled)
differential equation%
\begin{equation}
\dot{s}(t)=(A(t)-s(t)^{\top}A(t)s(t)\cdot I_{d})s(t), \label{hom_S}%
\end{equation}
which leaves $\mathbb{S}^{d-1}\subset\mathbb{R}^{d}$ invariant. This coincides
with the periodic differential equation obtained by projecting the homogeneous
part (\ref{hom}) to $\mathbb{S}^{d-1}$. Furthermore, the equator is invariant,
hence also the upper hemisphere $\mathbb{S}^{d,+}:=\{s=(s_{1},\ldots
,s_{d},s_{d+1})\in\mathbb{S}^{d}\left\vert s_{d+1}>0\right.  \}$ is invariant.
When the phases $\tau\in\lbrack0,T)$ are added to the states, the periodic
differential equations (\ref{hom}) and (\ref{hom_S}) induce autonomous
differential equations on $\mathbb{S}^{1}\times\mathbb{R}^{d}$ and
$\mathbb{S}^{1}\times\mathbb{S}^{d-1}$, resp.

A conjugacy of (autonomous) control systems%
\begin{align*}
\dot{x}(t)  &  =f(x(t),u(t))\text{ on }M\text{ with }u(t)\in U,\\
\dot{y}(t)  &  =g(y(t),u(t))\text{ on }N\text{ with }u(t)\in U,
\end{align*}
on manifolds $M$ and $N$ can be defined as a map $h:M\rightarrow N$ which
together with its inverse $h^{-1}$ is $C^{\infty}$ such that the trajectories
$\varphi(t;x_{0},u),t\in\mathbb{R}$, on $M$ and $\psi(t;y_{0},u),t\in
\mathbb{R}$, on $N$ with initial conditions $\varphi(0;x_{0},u)=x_{0}$ and
$\psi(0;y_{0},u)=y_{0}$ (assumed to exist) satisfy%
\[
h(\varphi(t;x_{0},u))=\psi(t;h(x_{0}),u)\text{ for all }t\in\mathbb{R}\text{
and }x\in M,u\in\mathcal{U}.
\]
Analogously, one can define conjugacies of differential equations. It is clear
that reachable sets, controllable sets, and control sets are preserved under conjugacies.

In the following, we slightly abuse notation by identifying vectors and their
transposes when it is clear from the context what is meant.

\begin{proposition}
\label{Proposition_e}(i) The map%
\[
e_{P}:\mathbb{S}^{1}\times\mathbb{R}^{d}\rightarrow\mathbb{S}^{1}%
\times\mathbb{S}^{d,+},\,(\tau,x)\mapsto\left(  \tau,\frac{(x,1)}{\left\Vert
(x,1)\right\Vert }\right)  =\left(  \tau,\frac{(x,1)}{\sqrt{1+\left\Vert
x\right\Vert ^{2}}}\right)
\]
is a conjugacy of the autonomized control system (\ref{aug1}) on
$\mathbb{S}^{1}\times\mathbb{R}^{d}$ and the restriction to $\mathbb{S}%
^{1}\times\mathbb{S}^{d,+}$ of the autonomized system induced by (\ref{5.3}).

(ii) The map $e_{\mathbb{S}}:\mathbb{S}^{1}\times\mathbb{S}^{d-1}%
\rightarrow\mathbb{S}^{1}\times\mathbb{S}^{d,0},\,(\tau,s)\mapsto(\tau
,s,0)\in\mathbb{S}^{1}\times\mathbb{S}^{d,0}$ is a conjugacy of the
autonomized differential equation induced by (\ref{hom_S}) on $\mathbb{S}%
^{1}\times\mathbb{S}^{d-1}$ and the restriction to $\mathbb{S}^{1}%
\times\mathbb{S}^{d,0}$ of the autonomized control system corresponding to
(\ref{5.3}).
\end{proposition}

\begin{proof}
(i) The map $e_{P}$ is $C^{\infty}$ (even analytic) with $C^{\infty}$ inverse
given by%
\[
\left(  e_{P}\right)  ^{-1}(\tau,s_{1},\ldots,s_{d},s_{d+1})=\left(
\tau,\frac{s_{1}}{s_{d+1}},\ldots,\frac{s_{d}}{s_{d+1}}\right)  \text{ for
}(\tau,s_{1},\ldots,s_{d},s_{d+1})\in\mathbb{S}^{1}\times\mathbb{S}^{d,+}.
\]
In fact, one verifies%
\[
e_{P}(\left(  e_{P}\right)  ^{-1}(\tau,s))=\left(  \tau,\frac{\left(
\frac{s_{1}}{s_{d+1}},\ldots,\frac{s_{d}}{s_{d+1}},1\right)  }{\sqrt
{1+\frac{s_{1}^{2}}{s_{d+1}^{2}}+\cdots+\frac{s_{d}^{2}}{s_{d+1}^{2}}}%
}\right)  =(\tau,s).
\]
The conjugacy property follows from
\[
\left(  \tau(t),s(t)\right)  =\left(  \tau(t),\frac{(x(t),1)}{1+\left\Vert
(x(t),1)\right\Vert }\right)  =e_{P}(\tau(t),x(t)),t\in\mathbb{R}.
\]

(ii) This trivially holds since the solutions on $\mathbb{S}^{d,0}$ are
obtained by adding the last component $0$ to the solutions on $\mathbb{S}%
^{d-1}$.
\end{proof}

Since the image of the map $e_{P}$ is contained in the (open) upper hemisphere
$\mathbb{S}^{d,+}$ a converging sequence of points $e_{P}(\tau_{k},x_{k})$ in
the image converges to an element of $\mathbb{S}^{1}\times\mathbb{S}^{d,0}$ if
and only if $\left\Vert x_{k}\right\Vert \rightarrow\infty$. Hence Proposition
\ref{Proposition_e} shows that the behavior near the equator reflects the
behavior near infinity.

Next we discuss the projection of the reachable and controllable sets to the
Poincar\'{e} sphere. Note that under the map $x\mapsto\frac{(x,1)}{\left\Vert
(x,1)\right\Vert }$ the origin $x=0\in\mathbb{R}^{d}$ is mapped to the north
pole $(0,1)\in\mathbb{S}^{d,+}\subset\mathbb{R}^{d}\times\mathbb{R}$. The
following theorem shows that for the autonomized system the closure of the
reachable set from the north pole intersects the equator in the image of the
center-unstable subbundle, and the closure of the controllable set to the
north pole intersects the equator in the image of the center-stable subbundle.
Furthermore the closure of the unique control set with nonvoid interior on
$\mathbb{S}^{1}\times\mathbb{S}^{d,+}$ intersects the equator in the image of
the center subbundle.

\begin{theorem}
\label{Theorem_ep}Suppose that the periodic system in (\ref{periodic0}) with
unconstrained controls is controllable.

(i) Then the projections to $\mathbb{S}^{1}\times\mathbb{S}^{d,+}$ of the
reachable and controllable sets, resp., of the autonomized system (\ref{aug1})
satisfy%
\begin{align*}
\overline{e_{P}(\mathrm{int}\mathbf{R}^{a}(0,0))}\cap\left(  \mathbb{S}%
^{1}\times\mathbb{S}^{d,0}\right)   &  =\overline{e_{P}(\mathcal{E}^{+,0}%
)}\cap\left(  \mathbb{S}^{1}\times\mathbb{S}^{d,0}\right)  ,\\
\overline{e_{P}(\mathrm{int}\mathbf{C}^{a}(0,0))}\cap\left(  \mathbb{S}%
^{1}\times\mathbb{S}^{d,0}\right)   &  =\overline{e_{P}(\mathcal{E}^{-,0}%
)}\cap\left(  \mathbb{S}^{1}\times\mathbb{S}^{d,0}\right)  .
\end{align*}

(ii) The induced system on $\mathbb{S}^{1}\times\mathbb{S}^{d,+}$ has a unique
control set with nonvoid interior given by $D_{P}^{a}=e_{P}(D^{a})$ satisfying%
\[
\overline{\mathrm{int}D_{P}^{a}}\cap(\mathbb{S}^{1}\times\mathbb{S}%
^{d,0})=\overline{e_{P}\left(  \mathcal{E}^{0}\right)  }\cap\left(
\mathbb{S}^{1}\times\mathbb{S}^{d,0}\right)  .
\]
In particular, $\mathrm{int}D^{a}$ is bounded if and only if $\overline
{\mathrm{int}D_{P}^{a}}\subset\mathbb{S}^{1}\times\mathbb{S}^{d,+}$.
\end{theorem}

\begin{proof}
(i) The conjugacy property from Proposition \ref{Proposition_e}(i) shows that
$e_{P}(\mathbf{R}^{a}(0,0))$ and $e_{P}(\mathbf{C}^{a}(0,0))$ are the
reachable and controllable set, resp., of $(0,(0,1))=e_{P}(0,0)$ in
$\mathbb{S}^{1}\times\mathbb{S}^{d,+}$, where $(0,1)$ is the north pole, and
$D_{P}^{a}:=e_{P}(D^{a})$ is the unique control set with nonvoid interior
satisfying $\mathrm{int}(e_{P}(D^{a}))=e_{P}(\mathrm{int}D^{a})$. The
inclusion $e_{P}(\mathcal{E}^{+,0})\subset\mathrm{int}\mathbf{R}^{a}(0,0)$
follows from Theorem \ref{Theorem_sub} implying the inclusion
\textquotedblleft$\supset$\textquotedblright. For the converse, let
$(\tau,z,0)\in\overline{e_{P}(\mathrm{int}\mathbf{R}^{a}(0,0))}\cap\left(
\mathbb{S}^{1}\times\mathbb{S}^{d,0}\right)  $. By Theorem \ref{Theorem_sub}
there are $b_{k}\in K_{\tau_{k}}^{-}$ and $x_{k}\in E_{\tau_{k}}^{+,0}$ with
$e_{P}(\tau_{k},b_{k}+x_{k})\rightarrow(\tau,z,0)$ and $\left\Vert
x_{k}\right\Vert \rightarrow\infty$. This implies $\left\Vert b_{k}%
+x_{k}\right\Vert \rightarrow\infty$ and%
\[
e_{P}(\tau_{k},b_{k}+x_{k})=\left(  \tau_{k},\frac{b_{k}}{\left\Vert
(b_{k}+x_{k},1)\right\Vert }+\frac{x_{k}}{\left\Vert (b_{k}+x_{k}%
,1)\right\Vert },\frac{1}{\left\Vert (b_{k}+x_{k},1)\right\Vert }\right)  .
\]
Using that the $b_{k}$ remain bounded, one finds%
\[
\frac{b_{k}}{\left\Vert (b_{k}+x_{k},1)\right\Vert }\rightarrow0,\,\frac
{1}{\left\Vert (b_{k}+x_{k},1)\right\Vert }\rightarrow0\text{, and }%
\frac{x_{k}}{\left\Vert (b_{k}+x_{k},1)\right\Vert }-\frac{x_{k}}{\left\Vert
(x_{k},1)\right\Vert }\rightarrow0\text{.}%
\]
Then it follows that
\[
e_{P}(\tau_{k},x_{k})=\left(  \tau_{k},\frac{(x_{k},1)}{\left\Vert
(x_{k},1)\right\Vert }\right)  \rightarrow(\tau,z,0).
\]
This shows that $(\tau,z,0)\in\overline{e_{P}(\mathcal{E}^{+,0})}\cap\left(
\mathbb{S}^{1}\times\mathbb{S}^{d,0}\right)  $. The assertions for the
controllable set follow similarly.

(ii) By Theorem \ref{Theorem_cs1} it follows that the unbounded part of
$\mathrm{int}D^{a}$ is $\mathcal{E}^{0}$.
\end{proof}

\begin{remark}
\label{Remark5.4}Theorem \ref{Theorem_ep}(i) shows that for the autonomized
system on the Poincar\'{e} sphere bundle $\mathbb{S}^{1}\times\mathbb{S}^{d}$
the closure of the reachable set from the north pole $e_{P}(0,0)=(0,(0,1))\in
\mathbb{S}^{1}\times\mathbb{S}^{d}$ intersects the \textquotedblleft
equator\textquotedblright\ $\mathbb{S}^{1}\times\mathbb{S}^{d,0}$ in the image
under $e_{P}$ of the center-stable subbundle $\mathcal{E}^{+,0}$. A closer
look at the dynamics on the equator reveals a finer picture: Consider the
Floquet bundles $\{(\tau,x)\in\mathbb{S}^{1}\times\mathbb{R}^{d}\left\vert
x\in L(\lambda_{j},\tau)\right.  \}$. The projected flow on the projective
bundle $\mathbb{S}^{1}\times\mathbb{P}^{d-1}$ goes from the projected Floquet
bundle for $\lambda_{j}$ to the projected Floquet bundles with $\lambda
_{i}>\lambda_{j}$. This can be made precise by some notions from topological
dynamics: the projected Floquet bundles form the finest Morse decomposition,
in particular, they coincide with the chain recurrent components (cf. Colonius
and Kliemann \cite[Section 7.2 and Theorem 8.3.3]{ColK14}). For the relation
to the flow on $\mathbb{S}^{1}\times\mathbb{S}^{d-1}$ one can prove that for
every chain recurrent component on $\mathbb{S}^{1}\times\mathbb{P}^{d-1}$
there are at most two chain recurrent components on $\mathbb{S}^{1}%
\times\mathbb{S}^{d-1}$ projecting to it. By Proposition \ref{Proposition_e}%
(ii) this also describes the flow on the \textquotedblleft
equator\textquotedblright\ $\mathbb{S}^{1}\times\mathbb{S}^{d,0}$. Examples
\ref{Example6.2} and \ref{Example6.3} illustrate some of these claims.
\end{remark}

\section{Examples\label{Section6}}

First we note the following consequence of Theorem \ref{Theorem_cs1}. In the
scalar case with $d=1$ one obtains from the inclusions in (\ref{D_a2}) that
one of the following cases holds: The set $\mathrm{int}D^{a}$ is contained
either in $\mathcal{K}^{-}$ or in $\mathcal{K}^{+}$ (if the Floquet exponent
is negative or positive, resp.) or $D^{a}=\mathcal{E}^{0}=\mathbb{S}^{1}%
\times\mathbb{R}$ (if $0$ is the Floquet exponent). In the first two cases
$D^{a}$ is bounded, in the third case it is unbounded.

\begin{example}
\label{Example6.1}Consider the periodic scalar example%
\begin{equation}
\dot{x}(t)=a(t)x(t)+u(t),\quad u(t)\in U=[-1,1], \label{6.1}%
\end{equation}
with $a(t):=-1$ for $t\in\lbrack0,1]$ and $a(t):=-2$ for $t\in(1,2]$ extended
to a $2$-periodic function on $\mathbb{R}$. Note that for $t\geq s,x_{0}%
\in\mathbb{R}$, and $u\in\mathcal{U}$ the solution is%
\[
\varphi(t;s,x_{0},u)=X(t,s)x_{0}+\int_{s}^{t}X(t,\sigma)u(\sigma)d\sigma\text{
with }X(t,s)=e^{\int_{s}^{t}a(\sigma)d\sigma}>0.
\]
The system with unconstrained controls is controllable, and the stable
subspace is $E_{\tau}^{-}=\mathbb{R}$ for all $\tau\in\lbrack0,2]$. Lemma
\ref{Lemma_point} implies that $\mathbb{S}^{1}\times\{0\}\subset
\mathrm{int}\mathbf{C}^{a}(0,0)$ for the autonomized system, and taking
$u\equiv0$ one sees that $\mathbb{S}^{1}\times\mathbb{R}=\overline
{\mathbf{C}^{a}(0,0)}$, hence $\mathbf{C}^{a}(0,0)=\mathbb{S}^{1}%
\times\mathbb{R}$. By Theorem \ref{Theorem_cs1} there is a unique control set
$D^{a}$ with nonvoid interior and $\mathbb{S}^{1}\times\{0\}\subset
\mathrm{int}D^{a}$. This yields $D^{a}=\overline{\mathbf{R}^{a}(0,0)}$. Recall
from Lemma \ref{Lemma3.3} that%
\begin{equation}
\mathbf{R}^{a}(0,0)=\{(\tau,x)\in\mathbb{S}^{1}\times\mathbb{R}^{d}\left\vert
x\in\mathbf{R}_{2\mathbb{N+\tau}}(0,0)\right.  \}. \label{6.2}%
\end{equation}
The solutions satisfy, for $t\geq s\geq0$ and $u\in\mathcal{U}$,%
\[
\varphi(t;s,x_{0},-1)\leq\varphi(t;s,x_{0},u)=-\varphi(t;s,-x_{0}%
,-u)\leq\varphi(t;s,x_{0},1).
\]
This implies that $\varphi(t;0,0,u)\leq\varphi(t;0,0,1)$ for all $t\geq0$ and
$u\in\mathcal{U}$. Since $U=-U$ the equation above with $x_{0}=0$ implies that
the reachable sets $\mathbf{R}_{t}(0,0)$ are symmetric around $0$. Together
with (\ref{6.2}) this shows that for the computation of $D^{a}$ it suffices to
determine $\overline{\mathbf{R}^{a}(0,0)}\cap\lbrack0,\infty)$\textbf{.} By
Proposition \ref{proposition_R} $\mathbf{R}_{2\mathbb{N+\tau}}(0,0)$ is
convex. Using $u\equiv0\,$\ and $u\equiv1$ one finds that $\mathbf{R}%
_{2k\mathbb{+\tau}}(0,0)\cap\lbrack0,\infty)=$ $[0,\varphi(2k+\tau;0,0,1)]$
for all $k\in\mathbb{N}$ and $\tau\in\mathbb{S}^{1}=[0,2)$.

\textbf{Claim:} For fixed $k\in\mathbb{N}$ the reachable sets $\mathbf{R}%
_{2k+\tau}(0,0)$ are increasing with $\tau\in\lbrack0,1]$ and decreasing with
$\tau\in\lbrack1,2)$. For fixed $\tau\in\lbrack0,2)$ they are increasing with
$k\in\mathbb{N}$ and they are given by%
\begin{equation}
\mathbf{R}_{2\mathbb{N}+\tau}(0,0)=\bigcup_{k\in\mathbb{N}}\mathbf{R}%
_{2k+\tau}(0,0)=\left(  -\frac{r(\tau)}{1-e^{-3}},\frac{r(\tau)}{1-e^{-3}%
}\right)  , \label{R01}%
\end{equation}
where%
\[
r(\tau):=\left\{
\begin{array}
[c]{ccc}%
\frac{1}{2}e^{-2-\tau}-e^{-3}-\frac{1}{2}e^{-\tau}+1 & \text{for} & \tau
\in\lbrack0,1]\\
\frac{1}{2}e^{2-2\tau}\left(  1-e^{-1}-e^{2\tau-5}\right)  +\frac{1}{2} &
\text{for} & \tau\in\lbrack1,2)
\end{array}
\right.  .
\]
Since in the proof of this claim we always take control $u\equiv1$ we suppress
this argument in $\varphi$. Let $x_{0}\in\mathbb{R}$ and compute for $\tau
\in\lbrack0,1]$ using $2$-periodicity%
\begin{align}
\varphi(2+\tau;\tau,x_{0})  &  =\varphi(2+\tau;2,\varphi(2;\tau,x_{0}%
))=\varphi(\tau;0,\varphi(2;\tau,x_{0}))\nonumber\\
&  =\varphi(\tau;0,\varphi(2;1,\varphi(1;\tau,x_{0})))=e^{-3}x_{0}+r(\tau).
\label{6.8}%
\end{align}
For $\tau\in\lbrack1,2]$ compute%
\begin{align*}
\varphi(\tau;0,0)  &  =\varphi(\tau;1,\varphi(1;0,0))=\frac{1}{2}e^{-2\tau
+2}-e^{-2\tau+1}+\frac{1}{2},\\
\varphi(2+\tau;\tau,x_{0})  &  =\varphi(2+\tau;3,\varphi(3;\tau,x_{0}%
))=\varphi(\tau;1,\varphi(3;\tau,x_{0}))=e^{-3}x_{0}+r(\tau).
\end{align*}
Repeated use of these formulas, periodicity, and induction show for
$k\in\mathbb{N}$ and $\tau\in\lbrack0,2]$
\begin{align}
&  \varphi(2(k+1)+\tau;0,0)=\varphi(2+\tau;\tau,\varphi(2k+\tau;0,0))=e^{-3}%
\varphi(2k+\tau;0,0)+r(\tau)\nonumber\\
&  =e^{-3(k+1)}\varphi(\tau;0,0)+\sum\nolimits_{j=0}^{k}e^{-3j}r(\tau).
\label{6.8b}%
\end{align}
Equation (\ref{6.8b}) implies $\lim_{k\rightarrow\infty}\varphi(2k+\tau
;0,0)=\frac{r(\tau)}{1-e^{-3}}$ proving (\ref{R01}).

The sets $\mathbf{R}_{2k+\tau}(0,0)$ are increasing with $k$ since%
\begin{align*}
&  \varphi(2(k+1)+\tau;0,0)-\varphi(2k+\tau;0,0)=(e^{-3(k+1)}-e^{-3k}%
)\varphi(\tau;0,0)+e^{-3k}r(\tau)\\
&  =\left\{
\begin{array}
[c]{lll}%
e^{-3k-3-\tau}\left(  \frac{1}{2}e^{3}+\frac{1}{2}e-1\right)  >0 & \text{for}
& \tau\in\lbrack0,1]\\
e^{-2\tau+1}\left[  e^{-3}\left(  \frac{1}{2}e-1\right)  +\frac{1}{2}\right]
>0 & \text{for} & \tau\in\lbrack1,2]
\end{array}
\right.  .
\end{align*}
The sets $\mathbf{R}_{2\mathbb{N}+\tau}(0,0)$ are increasing with $\tau
\in\lbrack0,1]$ since for $0\leq\sigma\leq\tau\leq1$
\[
\varphi(2k+\tau;0,0)=\varphi(\tau;0,\varphi(2k;0,0))=e^{-\tau}(\varphi
(2k;0,0)-1)+1\leq\varphi(2k+\sigma;0,0).
\]
Here we use that $e^{-\tau}\leq e^{-\sigma}$ and that for $x\geq1$ one has
$a(t)x+u\leq0$ for all $u\in U$ implying $\varphi(2k;0,0)-1\leq0$.

The sets $\mathbf{R}_{2\mathbb{N}+\tau}(0,0)$ are decreasing with $\tau
\in\lbrack1,2]$ since for $1\leq\sigma\leq\tau\leq2$%
\begin{align*}
\varphi(2k+\tau;0,0)  &  =\varphi(\tau;1,\varphi(2k+1;0,0))=e^{2-2\tau}%
\varphi(2k+1;0,0)+\int_{1}^{\tau}e^{-2(\tau-s)}ds\\
&  =e^{2-2\tau}(\varphi(2k+1;0,0)-\frac{1}{2})+\frac{1}{2}\leq\varphi
(2k+\sigma;0,0).
\end{align*}
Here we use $e^{2-2\tau}\leq e^{2-2\sigma}$ and $\varphi(2k+1;0,0)\geq
\varphi(1;0,0)=1-e^{-1}>\frac{1}{2}$.

Fig. 1 presents a sketch of the control set $D^{a}$ in $\mathbb{S}^{1}%
\times\mathbb{R}$.
\end{example}

The following two examples are autonomous two dimensional linear control
systems. Hence it is not necessary to autonomize the system and the results on
the control sets in $\mathbb{R}^{2}$ follow from Sontag \cite[Corollary
3.6.7]{Son98}. These examples serve as illustrations for the projection to the
Poincar\'{e} sphere.

\begin{example}
\label{Example6.2}Consider%
\begin{equation}
\dot{x}(t)=x(t)+u(t),\quad\dot{y}(t)=-y(t)+u(t), \label{Ex6.2b}%
\end{equation}
with $u(t)\in U=[-1,1]$. Here the origin is a saddle for the uncontrolled
system. For the control system induced on the Poincar\'{e} sphere
$\mathbb{S}^{2}$ a computation based on (\ref{5.3}) yields%
\begin{align}
\dot{s}_{1}  &  =\left[  1-s_{1}^{2}+s_{2}^{2}-u(s_{1}s_{3}+s_{2}%
s_{3})\right]  s_{1}+us_{3}\nonumber\\
\dot{s}_{2}  &  =\left[  -1-s_{1}^{2}+s_{2}^{2}-u(s_{1}s_{3}+s_{2}%
s_{3}))\right]  s_{2}+us_{3}\label{Ex6.2}\\
\dot{s}_{3}  &  =\left[  -s_{1}^{2}+s_{2}^{2}-u(s_{1}s_{3}+s_{2}s_{3})\right]
s_{3}.\nonumber
\end{align}
For $u=0$ the north pole $(0,0,1)$ is the only equilibrium, and for
$u\not =0\,$the equilibria move away from the north pole. Theorem
\ref{Theorem_cs1} implies that there is a unique control set $D^{a}%
\subset\mathbb{S}^{1}\times\mathbb{R}^{2}$ with nonvoid interior and that it
is bounded. By Theorem \ref{Theorem_ep}(ii) $D_{P}^{a}=e_{P}(D^{a})$ is the
unique control set with nonvoid interior on the upper hemisphere
$\mathbb{S}^{2,+}$. On the equator one has $s_{3}=0$ and the equation reduces
to%
\[
\dot{s}_{1}=2s_{2}^{2}s_{1},\quad\dot{s}_{2}=-2s_{1}^{2}s_{2}.
\]
This coincides with the projection of the homogeneous part of the original
equation in $\mathbb{R}^{2}$ onto the unit circle $\mathbb{S}^{1}$. The
equilibria are $(\pm1,0,0)$ and $(0,\pm1,0)$. Linearization on the equator
$\mathbb{S}^{2,0}$ yields in $e^{1}=(1,0,0)$ and $e^{2}=(0,1,0)$%
\[
\dot{x}_{2}=-2x_{2}\text{ and }\dot{x}_{1}=2x_{1},\text{ resp.}%
\]
If we linearize on the sphere $\mathbb{S}^{2}$ we have to linearize
(\ref{Ex6.2}) in $e^{1}$ and $e^{2}$ with respect to the second and third
arguments only. We obtain $\left(
\begin{array}
[c]{cc}%
\mp2 & u\\
0 & -1
\end{array}
\right)  $ with eigenvalues $\mp2$ and $\mp1$ with eigenvectors given by
$\left(  x,0\right)  ^{\top}$ and $\left(  \pm u\cdot x,x\right)  ^{\top
},\allowbreak x\not =0$, resp.

The orthogonal projection of the system on the upper hemisphere $\mathbb{S}%
^{2,+}$ to the unit disk yields the global phase portrait with control set
$e_{P}(D^{a})$ sketched in\textbf{ }Figure 2. Observe that near the equator
$s_{3}$ is close to $0$, hence the control vector field in (\ref{Ex6.2}) goes
to $0$ for $s_{3}\rightarrow0$.
\end{example}

\begin{remark}
Perko \cite{Perko} considers the differential equation (\ref{Ex6.2b}) with
$u=0$. In this case the formulas derived above coincide with his results. The
global phase portrait in Figure 2 is similar to \cite[Figure 5 on p.
275]{Perko} with the additional feature that around the north pole of
$\mathbb{S}^{2}$ the image of the control set occurs. Perko \cite{Perko}, as
well as Lefschetz \cite[pp. 202]{Lefschetz}, actually, does these computations
for differential forms, i.e., in the cotangent bundle of the sphere.
\end{remark}

The following example is a slight modification of Example \ref{Example6.2}. It
illustrates Remark \ref{Remark5.4} since the flow on the intersection of
$\overline{e_{P}(\mathbf{R}(0))}$ with the equator is nontrivial.

\begin{example}
\label{Example6.3}Consider the autonomous system given by%
\[
\dot{x}(t)=x(t)+u(t),\quad\dot{y}(t)=2y(t)+u(t),
\]
with $u(t)\in U=[-1,1]$. Note that for constant $u$ the equilibrium given by
$(-u,-u/2)$ is an unstable knot. Since the eigenvalues $1$ and $2$ are
positive, the control set $D^{a}$ with nonvoid interior is bounded. The
reachable set from the origin coincides with the unstable subspace and
satisfies $\mathbf{R}(0)=E^{+}=\mathbb{R}^{2}$. For the projection to the
Poincar\'{e} sphere one obtains%
\[
\overline{e_{P}(\mathbf{R}(0))}\cap\mathbb{S}^{2,0}=\overline{e_{P}(E^{+}%
)}\cap\mathbb{S}^{2,0}=\mathbb{S}^{2,0}.
\]
On the other hand, Proposition \ref{Proposition_e}(ii) shows that the flow on
the equator $\mathbb{S}^{2,0}$ is determined by the flow on the unit circle
$\mathbb{S}^{1}$ induced by the homogeneous part (with $u\equiv0$). The
Floquet subspaces $L(1)=\mathbb{R}\times\{0\}$ and $L(2)=\{0\}\times
\mathbb{R}$ are given by the eigenspaces and intersect $\mathbb{S}^{1}$ in the
equilibria $(\pm1,0)$ and $(0,\pm1)$, resp. All other points $s_{0}%
\in\mathbb{S}^{1}$ satisfy $\lim_{t\rightarrow-\infty}s(t,s_{0})=(\pm1,0)$ and
$\lim_{t\rightarrow\infty}s(t,s_{0})=(0,\pm1)$. The orthogonal projection of
the system on the upper hemisphere $\mathbb{S}^{2,+}$ to the unit disk yields
the global phase portrait with control set $e_{P}(D^{a})$ sketched in Figure 3.
\end{example}

\section{Controllability properties of quasi-affine systems\label{Section7}}

In this section we apply the results above to the study of controllability
properties for quasi-affine control systems of the form (\ref{qaffine1}).
Explicitly, system (\ref{qaffine1}) may be written as%
\begin{equation}
\dot{x}(t)=A_{0}x(t)+\sum_{i=1}^{p}v_{i}(t)A_{i}x(t)+B(v(t))u(t),\quad
(u,v)\in\mathcal{U}\times\mathcal{V}. \label{qaffine2}%
\end{equation}
We denote the solutions of (\ref{qaffine2}) with initial condition
$x(0)=x_{0}\in\mathbb{R}^{d}$ by $\psi(t;x_{0},u,v),\allowbreak\,t\in
\mathbb{R}$. The homogeneous part of (\ref{qaffine2}) is the bilinear control
system
\begin{equation}
\dot{x}(t)=A(v(t))x(t),\quad v\in\mathcal{V}, \label{bilinear}%
\end{equation}
and we denote the solutions of (\ref{bilinear}) with $x(0)=x_{0}$ by
$\psi_{\hom}(t;x_{0},v),\,t\in\mathbb{R}$. Control systems (\ref{qaffine2})
and (\ref{bilinear}) come with associated flows given by
\begin{align*}
\Psi &  :\mathbb{R}\times\mathcal{U}\times\mathcal{V}\times\mathbb{R}%
^{d}\rightarrow\mathcal{U}\times\mathcal{V}\times\mathbb{R}^{d}:\Psi
(t;u,v,x):=(u(t+\cdot),v(t+\cdot),\psi(t;x,u,v)),\\
&  \left.  \Psi_{\hom}:\mathbb{R}\times\mathcal{V}\times\mathbb{R}%
^{d}\rightarrow\mathcal{V}\times\mathbb{R}^{d}:\Psi_{\hom}(t;v,x):=(v(t+\cdot
),\psi_{\hom}(t;x,v)),\right.
\end{align*}
resp. Here $u(t+\cdot)(s):=u(t+s)$ and $v(t+\cdot)(s):=v(t+s),s\in\mathbb{R}$,
are the right shifts and $\mathcal{U}\subset L^{\infty}(\mathbb{R}%
,\mathbb{R}^{m})$ and $\mathcal{V}\subset L^{\infty}(\mathbb{R},\mathbb{R}%
^{p})$ are endowed with a metric for the weak$^{\ast}$ topology. Then
$\mathcal{U}$ and $\mathcal{V}$ are compact and chain transitive; cf. Colonius
and Kliemann \cite[Chapter 4]{ColK00} or Kawan \cite[Section 1.4]{Kawa13}. The
flow $\Psi_{\hom}$ is a continuous linear skew product flow on the vector
bundle $\mathcal{V}\times\mathbb{R}^{d}$ since (\ref{bilinear}) is
control-affine. On the other hand, the affine flow $\Psi$ on the vector bundle
$(\mathcal{U}\times\mathcal{V})\times\mathbb{R}^{d}$ is not continuous, in
general, even if we suppose that $B(v):=B_{0}+\sum_{i=1}^{p}v_{i}B_{i}$ with
$B_{0},B_{1},\ldots,B_{p}\in\mathbb{R}^{d\times m}$. In fact, if products
$v_{i}u_{j}$ occur on the right hand side of (\ref{qaffine2}), the system is
not control-affine, and hence continuity does not hold.

For any periodic $v\in\mathcal{V}$ one obtains a periodic linear control
system%
\begin{equation}
\dot{x}(t)=A(v(t))x(t)+B(v(t))u(t),\quad u\in\mathcal{U}. \label{periodic_v}%
\end{equation}
Fix a $T_{v}$-periodic control $v\in\mathcal{V}$ and parametrize the unit
circle $\mathbb{S}^{1}$ by $\tau\in\lbrack0,T_{v})$. A corresponding augmented
autonomous control system on $\mathbb{S}^{1}\times\mathbb{R}^{d}$ is defined
by%
\begin{equation}
\psi_{v}^{a}(t;(\tau_{0},x_{0}),u)=(t+\tau_{0}\operatorname{mod}T_{v}%
,\psi(t;x_{0},u,v(\tau_{0}+\cdot))),\quad u\in\mathcal{U}. \label{7.5}%
\end{equation}
The reachable set of $(\tau_{0},x_{0})\in\mathbb{S}^{1}\times\mathbb{R}^{d}$
is%
\[
\mathbf{R}_{v}^{a}(\tau_{0},x_{0}):=\{\psi_{v}^{a}(t;(\tau_{0},x_{0}%
),u)\left\vert t\geq0\text{ and }u\in\mathcal{U}\right.  \mathcal{\}}.
\]
Analogously, the controllable sets $\mathbf{C}_{v}^{a}(\tau_{0},x_{0})$ are
defined. If the system in (\ref{periodic_v}) without control restriction is
controllable, Theorem \ref{Theorem_cs1} shows that one finds a unique control
set $D_{v}^{a}=\overline{\mathbf{R}_{v}^{a}(0,0)}\cap\mathbf{C}_{v}^{a}(0,0)$
with nonvoid interior of the autonomized system (\ref{7.5}) and $\mathbb{S}%
^{1}\times\{0\}\subset\mathrm{int}D_{v}^{a}$.

Let $\pi_{2}:\mathbb{S}^{1}\times\mathbb{R}^{d}\rightarrow\mathbb{R}^{d}%
,~\pi_{2}(\tau,x)=x$ for $(\tau,x)\in\mathbb{S}^{1}\times\mathbb{R}^{d}$. The
following theorem establishes the existence of a control set for quasi-affine
systems, defined analogously as in Definition \ref{Definition_control_sets},
containing all $\pi_{2}(D_{v}^{a})$ for the control sets $D_{v}^{a}$ for
periodic $v\in\mathcal{V}$.

\begin{theorem}
\label{Theorem7.1}Suppose that the following assumptions hold:

(i) for every periodic $v\in\mathcal{V}$ the periodic linear system in
(\ref{periodic_v}) with unconstrained controls $u\in L^{\infty}(\mathbb{R}%
,\mathbb{R}^{m})$ is controllable;

(ii) the quasi-affine system (\ref{qaffine2}) is locally accessible, i.e.,
$\mathbf{R}_{\leq S}(x)$ and $\mathbf{C}_{\leq S}(x)$ have nonvoid interiors
for all $S>0$ and all $x\in\mathbb{R}^{d}$.

Then the quasi-affine system (\ref{qaffine2}) has a control set $D$ with
nonvoid interior such that for all periodic $v\in\mathcal{V}$ the control sets
$D_{v}^{a}$ of the autonomized periodic linear control system (\ref{7.5})
satisfy $\pi_{2}(D_{v}^{a})\subset D.$
\end{theorem}

\begin{proof}
Fix a $T_{v}$-periodic control $v\in\mathcal{V}$. The set $\pi_{2}(D_{v}^{a})$
is a neighborhood of $0\in\mathbb{R}^{d}$ and for all $x,y\in\pi_{2}(D_{v}%
^{a})$ there are $\tau_{x},\tau_{y}\in\mathbb{S}^{1}$ with $(\tau_{x}%
,x),(\tau_{y},y)\in D_{v}^{a}$ and $(\tau_{y},y)\in\overline{\mathbf{R}%
_{v}(\tau_{x},x)}$. This means that there are $t_{n}\geq0$ and $u_{n}%
\in\mathcal{U}$ with%
\[
\psi_{v}^{a}(t_{n};(\tau_{x},x),u_{n})=(t_{n}+\tau_{x}\operatorname{mod}%
T_{v},\psi(t_{n};x,u_{n},v(\tau_{x}+\cdot)))\rightarrow(\tau_{y},y)\text{ for
}n\rightarrow\infty.
\]
In particular, this shows that $y\in\overline{\mathbf{R}(x)}$, where
$\mathbf{R}(x)$ is the reachable set from $x$ of the quasi-affine system
(\ref{qaffine2}) given by%
\[
\mathbf{R}_{\leq S}(x):=\{\psi(t;x,u^{\prime},v^{\prime})\left\vert
\,t\in\lbrack0,S],~(u^{\prime},v^{\prime})\in\mathcal{U}\times\mathcal{V}%
\right.  \}\text{ and }\mathbf{R}(x):=\bigcup\nolimits_{S>0}\mathbf{R}_{\leq
S}(x).
\]
Define $D$ as the union of all sets $D^{\prime}$ satisfying $D^{\prime}%
\subset\overline{\mathbf{R}(x)}$ for all $x^{\prime}\in D^{\prime}$ and
containing $\pi_{2}(D_{v}^{a})$. We claim that $D\subset\overline
{\mathbf{R}(x)}$ for all $x\in D$. For the proof of the claim, let $x,y\in D$.
Then there are sets $D^{\prime}$ and $D^{\prime\prime}$ with $\pi_{2}%
(D_{v}^{a})\subset D^{\prime}\cap D^{\prime\prime}$ and $x\in D^{\prime},y\in
D^{\prime\prime}$. We know that $0\in\mathrm{int}\pi_{2}(D_{v}^{a})$. By local
accessibility of the quasi-affine system there is $S>0$ with $\varnothing
\not =\mathrm{int}\mathbf{C}_{\leq S}(0)\subset\pi_{2}(D_{v}^{a})\subset
D^{\prime}$. Then the inclusion $D^{\prime}\subset\overline{\mathbf{R}(x)}$
implies $0\in\mathbf{R}(x)$. Since $0,y\in D^{\prime\prime}$ the claim follows
from $y\in\overline{\mathbf{R}(0)}\subset\overline{\mathbf{R}(x)}$. Thus $D$
is a maximal set with the property that for all $x\in D$ one has
$D\subset\overline{\mathbf{R}(x)}$. Since $\mathrm{int}D\not =\varnothing$
Kawan \cite[Proposition 1.20]{Kawa13} implies that $D$ is a control set.

For every periodic control $v\in\mathcal{V}$ the projected set $\pi_{2}%
(D_{v}^{a})$ contains $0\in\mathbb{R}^{d}$. Thus the maximality property of
control sets implies that the control set $D$ is independent of $v$ and hence
contains $\pi_{2}(D_{v}^{a})$ for every periodic $v\in\mathcal{V}$.
\end{proof}

Next we show that under some additional assumptions the control set $D$ of the
quasi-affine system coincides (up to closure) with the union of the projected
control sets $D_{v}^{a}$. Thus the control set $D$ can be obtained by fixing
periodic controls $v$ and determining the control sets $D_{v}$ of the
corresponding autonomized systems (\ref{aug1}).

\begin{theorem}
\label{Theorem7.2}Suppose that the following assumptions hold:

(i) for every periodic $v\in\mathcal{V}$ the periodic linear system in
(\ref{periodic_v}) with unconstrained controls $u\in L^{\infty}(\mathbb{R}%
,\mathbb{R}^{m})$ is controllable;

(ii) the quasi-affine system (\ref{qaffine2}) is locally accessible;

(iii) for all periodic $v\in\mathcal{V}$ all Floquet exponents of the periodic
homogeneous part (\ref{bilinear}) of (\ref{qaffine2}) are different from $0$;

Then the quasi-affine system (\ref{qaffine2}) has a unique control set $D$
with nonvoid interior, and it satisfies%
\[
\overline{D}=\overline{\bigcup\nolimits_{v\in\mathcal{V}\text{ periodic}}%
\pi_{2}(\mathrm{int}D_{v}^{a})}.
\]

\end{theorem}

\begin{proof}
The control set $D$ from Theorem \ref{Theorem7.1} contains all $\pi
_{2}(\mathrm{int}D_{v}^{a})$. Let $E$ be an arbitrary control set with nonvoid
interior of (\ref{qaffine2}). By local accessibility, Colonius and Kliemann
\cite[Lemma 3.2.13(i)]{ColK00} shows that $\overline{E}=\overline
{\mathrm{int}E}$. Hence it suffices to prove that $\mathrm{int}E\subset
\overline{\bigcup\nolimits_{v\in\mathcal{V}\text{ periodic}}\pi_{2}%
(\mathrm{int}D_{v}^{a})}$, which also implies $E=D$. Fix a point $x_{0}%
\in\mathrm{int}E$. Denote by $\mathrm{int}_{\infty}(\mathcal{U})$ the interior
of $\mathcal{U}$ with respect to the $L^{\infty}$-norm, and note that
$\mathrm{int}_{\infty}(\mathcal{U})$ is dense in $\mathcal{U}$ in this norm.

\textbf{Claim.} For every $\varepsilon>0$ there are $y\in\mathbb{R}^{d},T>0$,
and $(u,v)\in\mathrm{int}_{\infty}(\mathcal{U)}\times\mathcal{V}$ with
$\left\Vert y-x_{0}\right\Vert <\varepsilon$ and $\psi(dT;y,u,v)=y$.

For the proof of the claim note first that by local accessibility \cite[Lemma
3.2.13(iii)]{ColK00} implies $\mathrm{int}E\subset\mathbf{R}(x_{0})$ and hence
there are $u^{0}\in\mathcal{U}$, $v\in\mathcal{V}$, and $T>0$ with
$x_{0}=\varphi(T;x_{0},u^{0},v)$. We may suppose that $u^{0}$ and $v$ are
$T$-periodic functions. Thus $T=T_{v}$ for the $T$-periodic control
$v\in\mathcal{V}$ and we obtain $x_{0}=\varphi(dT_{v};x_{0},u^{0},v)$. Since
$\mathrm{int}_{\infty}(\mathcal{U})$ is dense in $\mathcal{U}$ one finds for
all $\varepsilon>0$ a $dT_{v}$-periodic control $u$ with $\left\Vert
u-u^{0}\right\Vert _{L^{\infty}}<\varepsilon$. By the hyperbolicity assumption
(ii), Colonius, Santana, Setti \cite[Proposition 2.9(i)]{ColSS22} implies that
the $dT_{v}$-periodic inhomogeneous differential equation (\ref{periodic_v})
has a unique $dT_{v}$-periodic solution with initial value $y$ at time $0$.
With the principal fundamental solution denoted by $X_{v}(t,s)$ it is given by%
\[
y=\left[  I_{d}-X_{v}(dT_{v},0)\right]  ^{-1}\int_{0}^{dT_{v}}X_{v}%
(dT_{v},s)B(v(s))u(s)ds.
\]
By \cite[Proposition 2.9(iv)]{ColSS22} the initial values $y$ of these
periodic solutions converge to $x_{0}$ for $u$ converging to $u^{0}$ in
$L^{\infty}([0,dT_{v}],\mathbb{R}^{m})$. This proves the \textbf{claim}.

It remains to prove that $x_{0}$ is in $\overline{\pi_{2}(D_{v}^{a})}$. This
follows if we can show that $y\in\pi_{2}(D_{v}^{a})$ since $y$ is arbitrarily
close to $x_{0}$. Assumption (i) and Theorem \ref{Theorem_Brunovsky}(ii) imply
that for all points $z\in\mathbb{R}^{d}$ there is $u^{\prime}\in L^{\infty
}([0,dT_{v}],\mathbb{R}^{m})$ such that
\[
z=\psi(dT_{v};0,u^{\prime},v)=\int_{0}^{dT_{v}}X_{v}(dT_{v},s)B(v(s))u^{\prime
}(s)ds.
\]
Since $u\in\mathrm{int}_{\infty}(\mathcal{U})$ it follows that for all $z$ in
a neighborhood $N_{1}(y)$ of $y$ there is $u^{\prime}\in\mathcal{U}$ with%
\[
z-X_{v}(dT_{v},0)y=\int_{0}^{dT_{v}}X_{v}(dT_{v},s)B(v(s))u^{\prime
}(s)ds,\text{ hence }z=\psi(dT_{v};y,u^{\prime},v).
\]
With $dT_{v}=0\operatorname{mod}T_{v}$ this means by Lemma \ref{Lemma3.3} that
the points $(0,z)\in\mathbb{S}^{1}\times N_{1}(y)$ are contained in the
reachable set%
\[
\mathbf{R}_{v,dT_{v}}^{a}(0,y)=\{(0,y^{\prime})\left\vert y^{\prime}%
\in\mathbf{R}_{v,dT_{v}}(0,y)\right.  \}
\]
of the autonomized system (\ref{7.5}) for the $T_{v}$-periodic $v$. Applying
the same arguments to the time reversed system, one finds that all
$(0,z)\in\mathbb{S}^{1}\times\mathbb{R}^{d}$ with $z$ in a neighborhood
$N_{2}(y)$ of $y$ are in the controllable set $\mathbf{C}_{v,dT_{v}}^{a}%
(0,y)$. Every point $(0,z)$ with $z\in N_{1}(y)\cap N_{2}(y)$ can be steered
to $(0,y)$ and then to any other point in this intersection. This implies that
$(0,y)$ is in the interior of a control set of the autonomized system
(\ref{7.5}). The only control set with nonvoid interior of this system is
$D_{v}^{a}$, hence it follows that $(0,y)\in\mathrm{int}D_{v}^{a}$ and
$y\in\pi_{2}(\mathrm{int}D_{v}^{a})$ and concludes the proof.
\end{proof}

Similarly as the periodic linear system (\ref{affine}) also the quasi-affine
system (\ref{qaffine2}) can be projected to the upper hemisphere
$\mathbb{S}^{d,+}$ of the Poincar\'{e} sphere by a conjugacy $e_{P}%
^{0}:\mathbb{R}^{d}\rightarrow\mathbb{S}^{d,+}$. We obtain the following corollary.

\begin{corollary}
Under the assumptions of Theorem \ref{Theorem7.2} the control set $D$ of the
quasi-affine system (\ref{qaffine2}) projects to the unique control set
$e_{P}^{0}(D)$ with nonvoid interior for the system induced on $\mathbb{S}%
^{d,+}$.
\end{corollary}

\end{document}